\documentclass{amsart}

\usepackage{color,amssymb, comment, mathrsfs,tikz,upgreek,contour, soul}
\usetikzlibrary{matrix}
\usetikzlibrary{arrows}
\usepackage[all]{xy}
\newtheorem{theorem}{Theorem}[section]
\newtheorem{lemma}[theorem]{Lemma}
\newtheorem{proposition}[theorem]{Proposition}
\newtheorem{corollary}[theorem]{Corollary}

\usepackage[only,llbracket,rrbracket,llparenthesis,rrparenthesis]{stmaryrd} 
\usepackage{accsupp}

\theoremstyle{definition}
\newtheorem{definition}[theorem]{Definition}
\newtheorem{construction}[theorem]{Construction}
\newtheorem{example}[theorem]{Example}

\newtheorem{notation}[theorem]{Notation}

\theoremstyle{remark}
\newtheorem{remark}[theorem]{Remark}

\numberwithin{equation}{section}
\newcommand{\bbN}{\mathbb{N}}
\newcommand{\bbZ}{\mathbb{Z}}

\newcommand{\h}{\mathrm{H}}
\newcommand{\hI}{\hat{\mathbf{I}}}
\newcommand{\bfI}{\mathbf{I}}

\newcommand{\bfJ}{\mathbf{J}}
\newcommand{\bfK}{\mathbf{K}}
\newcommand{\kk}{\Bbbk}
\newcommand{\op}{\mathrm{op}}
\newcommand{\Ay}[2]{{#1}\!\left\langle{#2} \right\rangle}
\newcommand{\del}{\partial}
\newcommand{\cls}[1]{\operatorname{cls}(#1)}
\newcommand{\gdeg}[1]{\mathscr{G}(#1)}
\newcommand{\xra}{\xrightarrow}
\newcommand{\id}{\operatorname{id}}
\newcommand{\Diff}{\operatorname{Diff}}
\newcommand{\Der}{\operatorname{Der}}
\newcommand{\Ind}{\operatorname{Ind}}
\newcommand{\Ext}{\operatorname{Ext}}
\newcommand{\Lie}{\operatorname{Lie}}
\newcommand{\Hom}{\operatorname{Hom}}
\newcommand{\Aut}{\operatorname{Aut}}
\newcommand{\supp}{\operatorname{supp}}
\newcommand{\HH}[2]{\operatorname{H}_{#1}(#2)}
\newcommand{\q}{\mathbf{q}}
\newcommand{\eps}{\upchi}
\newcommand{\e}{\eps}
\newcommand{\qwedge}{\bigwedge\!\!{}^\q\;}

\newcommand{\bbinom}{\genfrac{[}{]}{0pt}{}}

\newcommand{\cdgardp}{\operatorname{DGA}_{\e}^{\Gamma}(R)}

\makeatletter
\DeclareFontEncoding{LS2}{}{\@noaccents}
\makeatother
\DeclareFontSubstitution{LS2}{stix}{m}{n}

\DeclareSymbolFont{largesymbolsstix}{LS2}{stixex}{m}{n}

\DeclareMathDelimiter{\lbrbrak}{\mathopen}{largesymbolsstix}{"EE}{largesymbolsstix}{"14}
\DeclareMathDelimiter{\rbrbrak}{\mathclose}{largesymbolsstix}{"EF}{largesymbolsstix}{"15}

\begin{document}

\title[DG algebra over skew polynomial rings]{Differential graded algebra over quotients of skew polynomial rings by normal elements}


\author{}
\address{}
\curraddr{}
\email{}
\thanks{}

\author{Luigi Ferraro and W. Frank Moore}
\address{}
\curraddr{}
\email{}
\thanks{}

\subjclass[2010]{Primary }

\date{}

\dedicatory{}

\commby{}

\begin{abstract}
Differential graded algebra techniques have played a crucial role in the 
development of homological algebra, especially in the study of homological 
properties of commutative rings carried out by Serre, Tate, Gulliksen,
Avramov, and others.
In this article, we extend the construction of the Koszul complex and acyclic closure to a more 
general setting.  As an application of our constructions, we shine some light on the structure
of the Ext algebra of quotients of skew polynomial rings by ideals generated by normal elements.
As a consequence, we give a presentation of the
Ext algebra when the elements generating the ideal form a regular sequence, generalizing a theorem of
Bergh and Oppermann.  It follows that in this case the Ext algebra
is noetherian, providing a partial answer to a question of Kirkman, Kuzmanovich and Zhang.
\end{abstract}

\maketitle

\section*{Introduction}


Differential graded (DG) algebra techniques have played a crucial role in the 
development of homological algebra, especially in the study of homological 
properties of commutative rings carried out by Serre \cite{Ser}, Tate \cite{Tate}, 
Gulliksen \cite{Gul}, Avramov \cite{IFR}, and others.  
Central to much of this work is the notion of Koszul complex, or more 
generally the process of adjoining variables to remove unwanted homology classes.  
The Shafarevich complex, introduced by Golod-Shafarevich \cite{GS}
and further studied by Golod \cite{Gol}, generalizes this notion to
an arbitrary associative algebra over a field filtered by a semigroup.  Unfortunately,
the DG algebras obtained using the Shafarevich construction can be far from minimal, and therefore
do not convey the amount of information one is accustomed to in the commutative case.

The first main result of this paper shows that under some hypotheses on the cycle(s) in question
(see Definition \ref{def:killable} and Proposition \ref{prop:stillNormal}),
one may adjoin a set of exterior or divided powers variables (rather than free)
to kill cycles.  In particular, this extends the notion of Koszul complex
to a broad class of rings of interest in noncommutative algebraic geometry, in such
a way that the DG algebras one obtains are minimal.
When restricted to the case of a sequence of skew commuting variables in a skew polynomial
ring, our construction differs from (but is inspired by) previous ones (cf. \cite{AGP})
in that it carries a natural DG algebra structure.

In the commutative case, a useful feature of adjoining variables to kill homology classes is that
one may always repeat this process.  That is, if one has adjoined variables in degree $n+1$
in order to kill homology classes in degree $n$, the cycles of degree $n+1$
representing the homology classes of the resulting complex can then be killed by variables
adjoined in degree $n+2$, and so on.   Unfortunately, the technical hypotheses required by our construction
may not be satisfied in the extension, and hence we may not be able to continue
this procedure in general.  However, this difficulty is no longer present if one works in the context of
\emph{color commutative} rings.  As such,  the remainder of the paper is devoted to applications of our 
construction in this setting; we recall the definition.


Let $\kk$ be a field, $G$ an abelian group, and let $\e : G \times G \to \kk^*$ be a skew
bicharacter.  A $\kk$-algebra $R$ is called \emph{$\e$-color commutative} (or simply color
commutative if $\e$ is understood) provided $A$ admits a $G$-grading
$R = \bigoplus_{\sigma \in G} R_\sigma$ such that if $x \in R_\sigma$ and $y \in R_\tau$, then
$xy = \e(\sigma,\tau)yx$ (see Definition \ref{def:colorComm}).
While these perhaps seem exotic at first glance, such $\kk$-algebras are essentially just quotients
of skew polynomial rings by an ideal generated by normal elements (see Proposition 
\ref{prop:colorCommSkew}).  The benefit of incorporating the skew bicharacter $\e$ into the study of 
such algebras is that they help to contain the proliferation of constants that appear when commuting 
elements past one another.

Our goal is to develop the differential graded framework over a color commutative finitely
generated connected graded $\kk$-algebra $R$ in a way which parallels the theory in the case of a
commutative noetherian local ring.  Our treatment is inspired heavily by the lecture notes of 
Avramov \cite{IFR} as well as the text of Gulliksen-Levin \cite{GulLev}.

Section \ref{sec:CDGA} is devoted to developing the required machinery of color DG algebras needed 
in later sections.  In Section \ref{sec:AcyClos} we recall the definition of color derivations (see 
\cite{Feld}) and use them to prove that an acyclic closure (see Definition \ref{def:acyclicClosure}) of $\kk$ 
over $R$ is minimal, generalizing a fundamental result of Gulliksen \cite{Gul}.  In particular, when
$R = Q/(f_1,\dots,f_c)$ for $Q$ a skew polynomial ring and $(f_1,\dots,f_c)$ a regular sequence of 
normal elements (we call such an algebra a \emph{skew complete intersection}), an acyclic closure may be 
obtained by adjoining variables only in homological degree one and two, in a manner analagous to 
that of the Tate resolution in the commutative case.  As a consequence we obtain a rational expression
for the Poincar\'e series of $\kk$ over $R$, see Corollary \ref{cor:PoincSeries}.

In Section \ref{sec:CDGADP} we introduce the category of color commutative
DG algebras with divided powers, and prove that an acyclic closure of $\kk$ over a noetherian connected graded color
$\kk$-algebra $R$ is unique up to isomorphism, generalizing a result of Gulliksen-Levin \cite{GulLev}.
Given that the acyclic closure $\Ay{R}{Y}$ of $\kk$ over $R$ is unique, it is natural to consider
the $R$-module of color derivations from $\Ay{R}{Y}$ to itself as an invariant of $R$.  Section 7 is devoted
to proving that it is a \emph{color DG Lie $R$-algebra}, which is a natural generalization of a DG Lie algebra
that incorporates the skew bicharacter into the skew symmetry and Jacobi identities \cite{Scheu} (see 
Definition \ref{def:CLA} for details).  Its homology is therefore a 
graded color Lie algebra which we call the \emph{homotopy color Lie algebra} of $R$, denoted by 
$\pi(R)$.  This object carries several gradings, and the dimension of the components record the 
number of variables adjoined to the acyclic closure in each respective degree, just as in the 
commutative case.

In Section \ref{sec:extAlgebra}, we generalize \cite[Theorem 3]{Sjo} and show that the Ext algebra
$\Ext_R(\kk,\kk)$ (where $\kk$ is viewed as a \emph{right} $R$-module, see
\ref{rem:leftExtIsRightExtOp}) is the universal enveloping algebra of $\pi(R)$,
proving a version of the classical Poincar\'e-Birkhoff-Witt theorem for $U(\pi(R))$
along the way.  As a corollary, we obtain that $\Ext_R(\kk,\kk)$ is a
\emph{graded color Hopf algebra}, which is a generalization of a graded Hopf algebra that
incorporates the skew bicharacter $\e$ into the bialgebra structure.

In Section \ref{sec:pi}, we compute the Lie bracket of elements in cohomological degree one in the homotopy
color Lie algebra, generalizing a result of Sj\"odin \cite[Theorem 4]{Sjo}.  When $R$ is a skew 
complete intersection the homotopy color Lie algebra is concentrated in cohomological degree one
and two and thus the Lie bracket computations provide a presentation of $\Ext_R(\kk,\kk)$ as a 
$\kk$-algebra.  We provide this presentation in Section \ref{sec:skewCIs}; our presentation extends 
the one given by Bergh-Opperman in \cite{BerOpp} which was obtained using other methods.
This presentation also shows that the Ext algebra of a skew complete intersection is noetherian,
giving a partial answer to a question of Kirkman-Kuzmanovich-Zhang in \cite{NnCommCI}.

\section{Background}

Let $\kk$ be a field.  Until further notice, all $\kk$-algebras in this paper 
are assumed to be unital and associative, and unadorned tensor products
are assumed to be defined over $\kk$.

\begin{definition} \label{def:gddAlg}
A $\kk$-algebra $R$ is \emph{graded} if $R = \bigoplus_{n \in \bbZ} R_n$,
where each $R_n$ is a $\kk$-module and $R_mR_n \subseteq R_{m+n}$.  A graded 
algebra is \emph{connected} if $R_0 = \kk$ and $R_i = 0$ for $i < 0$.
A graded $\kk$-algebra $R$ is \emph{bigraded} if each component $R_n$ has a further vector
space decomposition $R_n = \bigoplus_{m \in \bbZ} R_{m,n}$ such that
$R_{m,n}R_{k,l} \subseteq R_{m+k,n+l}$. In particular, this implies that
$R = \bigoplus_{m,n \in \bbZ} R_{m,n}$.
\end{definition}

In this paper, all objects will be bigraded.  We will call the first grading the homological
grading, and the second the internal grading.  For example, a connected graded algebra
$R$ may be considered a bigraded algebra by concentrating it in homological degree zero.
The homological grading is the one relevant to many of the constructions that follow in this 
paper, but we will also need the internal grading for arguments that involve the minimality
of certain constructions.  We will often suppress the internal grading from our notation,
however, see Remark \ref{rem:threeGradings2}.

\begin{definition} \label{def:bihom}
Let $R$ be as in Definition \ref{def:gddAlg}. We say an element $x\in R$ is \emph{homogeneous} of (internal) degree $n$, and write $\deg x=n$, if $x\in R_n$. We say an element $x$ is \emph{bihomogeneous}
of (internal) degree $n$ and homological degree $m$, and write $\deg x=n$ and $|x|=m$, if $x\in R_{m,n}$.
\end{definition}

\begin{definition} \label{def:dga}
Let $R$ be a graded $\kk$-algebra.  A \emph{differential graded (DG) $R$-algebra} $A$
is a bigraded unital associative $\kk$-algebra with $R \subseteq A_0$
equipped with a graded $R$-linear differential $\del$ of homological degree $-1$ such that
$\del^2 = 0$, and such that the Leibniz rule holds:
$$\del(ab) = \del(a)b + (-1)^{|a|}a\del(b).$$
We denote the underlying $R$-algebra of $A$ as $A^\natural$.
Note that every graded $\kk$-algebra may be considered a DG algebra with trivial 
differential and homological grading.  One may also consider
DG modules over $A$, cf. \cite{IFR}.
\end{definition}

The Leibniz rule shows that the cycles $Z(A)$ form a graded $R$-subalgebra of 
$A$, and that the boundaries $B(A)$ are a two-sided ideal in $Z(A)$. 
Therefore $H(A) = Z(A)/B(A)$ is a graded $H_0(A)$-algebra.

\section{The process of killing a graded normal cycle} \label{sec:killCycle}

In a manner similar to the commutative case,  we would like to adjoin an exterior or
divided powers variable to a DG algebra in order to
kill a cycle in homology.  It turns out that the process of killing a
\emph{graded normal cycle} (whose definition follows) is quite similar to the 
commutative case.

\begin{definition}
Let $A$ be a bigraded $\kk$-algebra.  A bihomogeneous
element $z \in A$ is said to be \emph{graded normal} if
there exists a bigraded automorphism $\sigma$ of $A$ such that for
all bihomogeneous elements $y \in A$, one has
$zy = (-1)^{|y||z|}\sigma(y)z$ and $\sigma(z) = z$.
In such a case, we call such an automorphism $\sigma$ a
\emph{normalizing automorphism} of $z$.  
\end{definition}

Note that `graded normal' agrees with the usual definition of normal, except
that a normalizing automorphism associated to a graded normal element may 
differ from the usual one by a sign when the normal element is of odd 
homological degree.

\begin{definition} \label{def:killable}
Let $A$ be a DG algebra.  A cycle $z$ is called \emph{killable} if
it is a graded normal cycle whose normalizing automorphism
is a chain map.  If the homological degree of $z$ is odd, we assume in 
addition that $z^2 = 0$.  
\end{definition}

In order to adjoin a variable, we first must define the underlying algebra of
the extension.  To do this, we use twisted tensor products.

\begin{definition}
Let $A$ and $B$ be $\kk$-algebras with multiplication maps $\mu_A$ and $\mu_B$, 
respectively.  A $\kk$-linear homomorphism $\tau:B \otimes A \rightarrow A \otimes B$
is a \emph{twisting map} provided $\tau(b \otimes 1_A)=1_A \otimes b$ and 
$\tau(1_B \otimes a) = a \otimes 1_B$ for all $a \in A$, $b \in B$.
A multiplication on $A \otimes B$ is then given
by $\mu_\tau := (\mu_A \otimes \mu_B) \circ (\id_A \otimes \tau \otimes \id_B)$.
By \cite[Proposition 2.3]{CSV}, $\mu_\tau$ is associative if and only if 
$$\tau \circ (\mu_B \otimes \mu_A) 
	= \mu_\tau \circ (\tau \otimes \tau) \circ (\id_B \otimes \tau \otimes \id_A)$$
as maps $B \otimes B \otimes A \otimes A \rightarrow A \otimes B$.
The pair $(A \otimes B,\mu_\tau)$ is
a \emph{twisted tensor product} of $A$ and $B$, 
denoted by $A \otimes^\tau B$.
\end{definition}

We first treat the case of killing an even degree cycle.

\begin{construction} \label{cons:adjVarOdd}
Let $A$ be a DG $\kk$-algebra, and let $z$ be a killable cycle with
$\sigma$ a normalizing automorphism of $z$.

We construct a DG $\kk$-algebra $\Ay{A}{y}$ by setting
$\Ay{A}{y}^\natural = A \otimes^\tau \Ay{\kk}{y}$,
where $\Ay{\kk}{y}$ is the exterior algebra on a variable $y$ in homological 
degree $|z|+1$, $\tau$ is given by
$$\tau( (b + cy) \otimes a ) = a \otimes b + (-1)^{|a|} c\sigma(a) \otimes y,\quad a\in A, b,c\in\kk,$$
and $a$ is bihomogeneous.
We remark that in $\Ay{A}{y}$, one has $ya = (-1)^{|a|}\sigma(a)y$ for
all bihomogeneous $a$.  Note that this makes $\Ay{A}{y}$ a free left (and right)
$A$-module.  The differential of $\Ay{A}{y}$ is given by
$$\del(a_1 + a_2y) = \del(a_1) + \del(a_2)y + (-1)^{|a_2|}a_2z.$$
\end{construction}

\begin{proposition} \label{prop:adjVarOddDGA}
Construction \ref{cons:adjVarOdd} gives $\Ay{A}{y}$ the structure of a $DG$
$\kk$-algebra.
\end{proposition}

\begin{proof}
First, we check that $\del^2 = 0$.  Indeed, for $a = a_1 + a_2y$ bihomogeneous of homological degree $|a|$,
one has
\begin{eqnarray*}
\del(\del(a_1 + a_2y)) & = & \del(\del(a_1) + \del(a_2)y + (-1)^{|a_2|}a_2z) \\
                   & = & (-1)^{|a_2|}\del(a_2)z + (-1)^{|a_2|-1}\del(a_2)z = 0.
\end{eqnarray*}
Next, we verify the Leibniz rule.  Let $a = (a_1 + a_2y)$ and $b = (b_1 + b_2y)$ be bihomogeneous elements of $\Ay{A}{y}$.
Note that since $y$ is of odd degree, $|a_2|$ and $|a|$ have opposite parity, so that $(-1)^{|a_2|} = (-1)^{|a|-1}$.
Hence
\begin{eqnarray*}
\del(a_1 + a_2y)(b_1 + b_2y)
   & = & (\del(a_1) + \del(a_2)y + (-1)^{|a|-1}a_2z)(b_1 + b_2y) \\
   & = & \del(a_1)b_1 + \del(a_1)b_2y + (-1)^{|b|}\del(a_2)\sigma(b_1)y + \\
   &   & (-1)^{|a|-1}a_2zb_1 + (-1)^{|a|-1}a_2zb_2y.
\end{eqnarray*}
Similarly, one has
\begin{eqnarray*}
(a_1 + a_2y)\del(b_1 + b_2y)
   & = & (a_1 + a_2y)(\del(b_1) + \del(b_2)y + (-1)^{|b|-1}b_2z) \\
   & = & a_1\del(b_1) + a_1\del(b_2)y + (-1)^{|b|-1}a_1b_2z + \\
   &   & (-1)^{|b|-1}a_2\sigma(\del(b_1))y + a_2\sigma(b_2)zy.
\end{eqnarray*}
Finally, we consider $\del$ of the product $ab$:
\begin{eqnarray*}
\del((a_1+a_2y)(b_1+b_2y)) & = & \del(a_1b_1 + (a_1b_2 + (-1)^{|b|}a_2\sigma(b_1))y) \\
   & = & \del(a_1)b_1 + (-1)^{|b|}a_1\del(b) + \del(a_1b_2 + (-1)^{|b|}a_2\sigma(b_1))y + \\
   &   & (-1)^{|a|+|b|-1}(a_1b_2 + (-1)^{|b|}a_2\sigma(b_1))z \\
   & = & \del(a_1)b_1 + (-1)^{|b|}a_1\del(b) + \del(a_1)b_2y + (-1)^{|a|}a_1\del(b_2)y + \\
   &   & (-1)^{|b|}\del(a_2)\sigma(b_1)y + (-1)^{|a|+|b|-1}a_2\del(\sigma(b_1))y + \\
   &   & (-1)^{|a|+|b|-1}a_1b_2z + (-1)^{|a|-1}a_2\sigma(b_1)z.
\end{eqnarray*}
That $\del$ satisfies the Leibniz rule now follows from the fact that $\sigma$ is a chain map
and that $z$ is graded normal.
\end{proof}

One also has a similar construction when $z$ is a graded normal cycle of odd
degree.

\begin{construction} \label{cons:adjVarEven}
Consider the same setup as in Construction \ref{cons:adjVarOdd}, except 
that $z$ is killable of odd degree.

In this case, let $\Ay{A}{y}^\natural = A \otimes^\tau \Ay{\kk}{y}$,
where $\Ay{\kk}{y}$ is the divided powers algebra on a variable $y$
in degree $|z|+1$, and $\tau$ is given by
$$\tau( \sum_i c_iy^{(i)} \otimes a) = \sum_i c_i\sigma^i(a) \otimes y^{(i)}.$$
Recall that the divided powers algebra $\Ay{\kk}{y}$ has $\kk$-basis the
set of all ``divided powers'' $y^{(i)}$ with multiplication given by $y^{(i)}y^{(j)} = \binom{i+j}{i}y^{(i+j)}$, for $i,j$ nonnegative integers.
As is customary, we let $y^{(0)} = 1$.  We remark that in $\Ay{A}{y}$, one has
$y^{(i)}a = \sigma^i(a)y^{(i)}$ for all bihomogeneous $a$.  Note that again,
$\Ay{A}{y}$ is a free left (and right) $A$-module.
The differential of $\Ay{A}{y}$ is given by
$$\del(a_0 + \sum_{i \geq 1} a_i y^{(i)}) =
  \del(a_0) + \sum_{i \geq 1}(\del(a_i)y^{(i)} + (-1)^{|a_i|}a_izy^{(i-1)}).$$
\end{construction}

A computation similar to the proof of Proposition \ref{prop:adjVarOddDGA} shows
that the following proposition holds.

\begin{proposition}
Construction \ref{cons:adjVarEven} gives $\Ay{A}{y}$ the structure of a
$DG$ $\kk$-algebra.
\end{proposition}

If we wish to include the differential of $y$ as part of the notation for $\Ay{A}{y}$,
we write $\Ay{A}{y\mid\del(y) = z}$.

\begin{definition}
We call an extension $A \subseteq B$ of DG algebras obtained by
successive application of either Construction \ref{cons:adjVarOdd} or 
\ref{cons:adjVarEven} is called a \emph{semi-free extension} of $A$.
We denote such an extension by $\Ay{A}{Y}$ where $Y$ is the set
of variables we have adjoined.

For such an extension and a total order $<$
on the exterior and divided powers variables $Y$, a monomial in the variables $Y$
is said to be in \emph{normal form with respect to $<$} if the variables appearing in it
are written in increasing order with respect to $<$.  If the ordering
on $Y$ is understood, we say that such a monomial in $Y$ is in normal form.
\end{definition}

Now suppose one has two killable cycles $z_1$ and $z_2$ that we wish to remove
in homology.  In order to iterate the above procedure, we must ensure that $z_2$
remains graded normal in the semi-free extension used to kill the cycle
$z_1$.  This is achieved in the next proposition under a skew commuting
hypothesis.

\begin{proposition}
\label{prop:stillNormal}
Let $A$ be a DG algebra and let $z_1,z_2$ be killable cycles
with normalizing automorphisms  $\sigma_1,\sigma_2$ respectively.
If $\sigma_1\sigma_2=\sigma_2\sigma_1$ and $z_1,z_2$ skew commute,
then $z_2$ is a killable cycle in $\Ay{A}{y\mid \partial (y)=z_1}$.
\end{proposition}
\begin{proof}
We set the notation $z_1z_2=(-1)^{|z_1||z_2|}qz_2z_1$ and prove the theorem in the case $|z_1|$ even, the odd case is similar.
We define the following map
\begin{alignat*}{4}
\tilde{\sigma}_2:\;&\Ay{A}{y}&\longrightarrow&\Ay{A}{y}\\
&a+by\;&\longmapsto&\sigma_2(a)+q^{-1}\sigma_2(b)y,
\end{alignat*}
and prove that $\tilde{\sigma}_2$ is a normalizing chain automorphism for $z_2$ in $\Ay{A}{y}$. Let $a+by$ be a bihomogeneous element of $\Ay{A}{y}$, we check that $z_2$ is graded normal of $z_2$ in $\Ay{A}{y}$:
\begin{align*}
z_2(a+by)&=z_2a+z_2by\\
&=(-1)^{|z_2||a|}\sigma_2(a)z_2+(-1)^{|z_2|(|a|-1)}\sigma_2(b)z_2y\\
&=(-1)^{|z_2||a|}\sigma_2(a)z_2+(-1)^{|z_2||a|}q^{-1}\sigma_2(b)yz_2\\
&=(-1)^{|z_2||a|}\tilde{\sigma}_2(a+by)z_2.
\end{align*}
It is a straightforward verification that $\tilde{\sigma}_2$ is an automorphism of the
DG algebra $\Ay{A}{Y}$.

\end{proof}

As an application of the process of killing cycles, we develop the Koszul
complex of a skew commuting sequence of normal elements with commuting normalizing
automorphisms in an arbitrary connected graded algebra.

\begin{construction} \label{Kos}
Let $R$ be a connected graded algebra, considered as a DG algebra $A^0$ concentrated
in homological degree zero with trivial differential.  Let $f_1,\dots,f_c$
be a skew commuting sequence of normal elements of $R$, and let $\sigma_i$
denote a normalizing automorphism of $f_i$.  Finally, assume that
$\sigma_i\sigma_j = \sigma_j\sigma_i$ for all $i,j$. 

For $i = 1,\dots,c$, set $A^i = \Ay{A^{i-1}}{y_i~|~\del (y_i) = f_i}$.  By Proposition
\ref{prop:stillNormal}, $f_{i+1},\dots,f_c$ are normal in $A^i$ for each $i$, which
allows for Construction \ref{cons:adjVarOdd} to continue after adjoining the
$i^\text{th}$ variable.  We define the skew Koszul complex of $f_1,\dots,f_c$
over $R$ to be the DG algebra $A^c$, and denote it by $K^R(\mathbf{f})$.

It is clear that $K^R(\mathbf{f})$ is a free left and right $R$-module
with basis given by $1 \in R = K^R(\mathbf{f})_0$ together with monomials of the form
$y_{i_1}\cdots y_{i_r}$ for $1 \leq i_1 < i_2 < \cdots < i_r \leq c$.  The differential
on this basis is:
$$
\del(y_{i_1}\cdots y_{i_r})=\sum_{j=1}^r(-1)^{j-1}f_{i_j}\prod_{s=1}^{j-1}r_{i_s,i_j }y_{i_1}\cdots\hat{y}_{i_j}\cdots y_{i_r},
$$
where $f_if_j = r_{i,j}f_jf_i$.  Note that the differential is both left
and right $R$-linear by the Leibniz rule and that this differential differs
from the one given in \cite{AGP}.
\end{construction}

Next, we recall the definition of skew polynomial ring.

\begin{definition}
Let $\q=(q_{i,j})$ be a $n\times n$ matrix with entries in $\kk$ such that $q_{i,i}=1$ for all $i=1,\ldots,n$ and $q_{i,j}=q^{-1}_{j,i}$ for all $i,j=1,\ldots,n$. Then the \emph{skew polynomial ring} associated to the matrix $\q$ is
\[
\kk_\q[x_1,\ldots,x_n]=\frac{T(\kk x_1\oplus\cdots\oplus\kk x_n)}{(x_ix_j-q_{i,j}x_jx_i\;\mathrm{for\;all\;}i,j=1,\ldots,n)},
\]
where $T(\kk x_1\oplus\cdots\oplus\kk x_n)$ is the tensor algebra of $\kk x_1\oplus\cdots\oplus\kk x_n$.
It is clear that each $x_i$ is normal in $\kk_\q[x_1,\dots,x_n]$, and that the normalizing automorphisms
of the variables commute with one another.
\end{definition}

More generally (see Section \ref{sec:colorComm}), we will show that \emph{any} pair of normal elements
in a skew polynomial ring skew commute and their normalizing automorphisms commute with one another.

\begin{definition} \label{def:koszulComplex}
Let $Q = \kk_\q[x_1,\dots,x_n]$ and $R = Q/I$ for some homogeneous ideal $I$ of $Q$, and denote by $\bar{x}_1,\ldots,\bar{x}_n$ the images of the variables of $Q$ in $R$.
We denote the Koszul complex of $\bar{x}_1,\dots,\bar{x}_n$ over $R$ by $K^R$
and call it the \emph{Koszul complex of $R$.}
\end{definition}

Next we examine the effect that adjoining a variable has on $H(A)$.  First,
we must contend with a notion of regularity, suitably modified for the DG
setting.

\begin{definition}
Let $A$ be a graded algebra, and let $z \in A$ be a bihomogeneous normal 
element.  We denote by $(0:_Az)$ the left annihilator of $z$ in $A$.  If $|z|$ is even, we say $z$
is \emph{regular} provided $(0:_A z) = 0$. If $|z|$ is odd and $z^2=0$, then we say that $z$ is
regular if $(0:_A z) = zA$.
\end{definition}
\begin{remark}
It is worth noticing that in the previous definition, the left annihilator of $z$ is zero (even case)
or $zA$ (odd case) if and only if the right annihilator of $z$ is zero (even case) or $zA$ (odd case).
\end{remark}

\begin{theorem}\label{thm:RegEl}
Let $A$ be a DG algebra, and let $z$ be a killable cycle of degree $d \geq 0$
such that $w = \cls{z}$ is regular.  Then there is a canonical isomorphism
$\frac{H(A)}{wH(A)} \to H(\Ay{A}{y})$.
\end{theorem}

\begin{proof}
We follow the treatment given in \cite[\S 6.1]{IFR}.
Let $\sigma$ be a normalizing chain automorphism associated to $z$.
We consider the case when $d$ is even and odd separately.  If $d$ is even,
define $\vartheta : \Ay{A}{y} \to \Ay{A}{y}$ by
$\vartheta(a + by) = (-1)^{|b|}\sigma^{-1}(b)$.  A straightforward verification shows
that $\vartheta$ is a chain map of degree $-d-1$.  This gives a short exact
sequence of DG algebras
$$0 \to A \xra{\iota} \Ay{A}{y} \xra{\vartheta} A \to 0.$$
Computing the connecting map $\eth$, one sees that for $b \in H(A)$,
$\eth(b) = \sigma(b)w = wb$, so that the connecting map is left multiplication by $w$.  
Therefore the homology long exact sequence is
$$\cdots \to H_{n-d}(A) \xra{w\cdot} H_n(A) \xra{H_n(\iota)} H_n(\Ay{A}{y})
  \xra{H_n(\vartheta)} H_{n-d-1}(A) \to \cdots$$
which implies the theorem.

When $d$ is odd, we define
$\vartheta(\sum_i a_i y^{(i)}) = \sum_ia_iy^{(i-1)}$, which is again
a chain map of degree $-d-1$ and which in turn gives a short exact sequence
of DG algebras
$$0 \to A \xra{\iota} \Ay{A}{y} \xra{\vartheta} \Ay{A}{y} \to 0.$$
While multiplication by $w$ does not appear in the long exact
sequence associated to the long exact sequence above, we can determine the effect of
killing the homology class $w$ using the spectral sequence associated to the filtration
$F^p \Ay{A}{y} = \sum_{i \leq p} Ay^{(i)}$, c.f. \cite[Section 5.4]{Wei}.

By definition, ${}^0\!E_{p,q} = F_p \Ay{A}{y}_{p+q}/F_{p-1} \Ay{A}{y}_{p+q}
\cong  A_{q-dp}y^{(p)}$, and the differential ${}^0\!d_{p,q}$ sends $ay^{(p)}$ to 
$\del(a)y^{(p)}$, so that we have ${}^1\!E_{p,q} = H_{q-dp}(A)$.  The differential
${}^1\!d_{p,q}$ in turn is induced by the differential on $\Ay{A}{y}$ and sends
$ay^{(p)}$ to $(-1)^{|a|}azy^{(p-1)}$.  Therefore ${}^2\!E_{p,q}$ is the homology of the
complex
$$H_{q - d(p+1)}(A) \xra{\cdot w} H_{q - dp}(A) \xra{\cdot w} H_{q - d(p-1)}(A).$$
Therefore for all $q$ one has
$${}^2\!E_{0,q} = \frac{H_{q}(A)}{wH_{q-d}(A)}\qquad\text{and}\qquad
{}^2\!E_{p,q} = \frac{(0:_{H(A)}w)_{q-pd}}{wH_{q-(p+1)d}(A)}~~\text{when}~p\geq 1.$$
If $w$ is regular, then we have ${}^2\!E_{p,q} = 0$ for $p \geq 1$, giving us the
desired conclusion, since the spectral sequence converges to $H(\Ay{A}{y})$.
\end{proof}

A straightforward consequence of Theorem \ref{thm:RegEl} is
\begin{corollary}
If $f_1,\ldots,f_c$ is a regular sequence of normal elements in the skew polynomial ring $Q$, then
$K^Q(f_1,\ldots,f_c)$ is a $Q$-resolution of $Q/(f_1,\ldots,f_c)$.
\end{corollary}
\section{Color commutative rings}
\label{sec:colorComm}

As mentioned in the previous section, in order to iterate the procedure of
adjoining variables, one requires the cycles to skew commute, and that their
normalizing automorphisms commute.  This leads one to consider algebras for which this
hypothesis is always satisfied.  In this section, we introduce such a class of 
algebras which in the end turns out to be familiar.

For the rest of the paper, $G$ will denote an abelian group, written multiplicatively,
and we will use $e_G$ to denote the identity of $G$.

\begin{definition}
A function $\eps : G \times G \to \kk^*$ is called a \emph{skew bicharacter}
provided for all $\alpha,\beta,\sigma\in G$, one has
\begin{eqnarray*}
\eps(\sigma,\alpha)\eps(\sigma,\beta) & = & \eps(\sigma,\alpha\beta) \\
\eps(\alpha,\beta) & = & \eps(\beta,\alpha)^{-1} \\
\eps(\alpha,\alpha) & = & 1
\end{eqnarray*}
\end{definition}
Note that the last condition is not typically part of the definition, but
we require it for the proof of Proposition \ref{prop:colorDGAy}.
In particular, note that for all $\alpha,\beta \in G$, one has
$\eps(\alpha,\beta) = \eps(\alpha,\beta^{-1})^{-1}$ and $\eps(e_G,\beta) = 1$.
Our interest in skew bicharacters comes from the fact that an algebra that is
generated by skew commuting elements defines a skew bicharacter on a subgroup
of the automorphism group of the algebra, as seen in the next example.

\begin{example} \label{ex:skewColor}
Let $Q = \kk_\q[x_1,\dots,x_n]$ be a skew polynomial ring.  Then a basis of $Q$ consists of
the set of all (ordered) monomials in these variables, and are hence normal
as well.  Let $G$ be the subgroup of $\Aut(Q)$ generated by the normalizing
automorphisms associated to the variables.  Then it follows that $G$ is an abelian
group, and that $Q$ admits a $G$-grading by associating to a monomial of
$Q$ its corresponding normalizing automorphism.  

We use the skew commutativity of the variables to define a skew bicharacter on $G$.
To do this, first suppose that $\sigma$ and $\tau$ are normalizing automorphisms of 
monomials $m_\sigma$ and $m_\tau$ respectively.  Then
$m_\sigma m_\tau = q_{\sigma,\tau} m_\tau m_\sigma$ 
for some $q_{\sigma,\tau} \in \kk^*$.  We define $\eps(\sigma,\tau) = q_{\sigma,\tau}$.

For general $\sigma,\tau \in G$, we have that $\sigma = \sigma_1\sigma_2^{-1}$
for some $\sigma_1,\sigma_2$ such that each $\sigma_i$ is the normalizing automorphism
of a monomial.  Likewise, $\tau = \tau_1\tau_2^{-1}$ for some $\tau_1,\tau_2$ where
$\tau_j$ is the normalizing automorphism of a monomial.  We then define
$$\eps(\sigma,\tau) = \eps(\sigma_1,\tau_1)\eps(\sigma_2,\tau_1)^{-1}
\eps(\sigma_1,\tau_2)^{-1}\eps(\sigma_2,\tau_2).$$
A routine verification shows that $\eps$ is well-defined and that it
is a skew bicharacter on $G$.
\end{example}

To capture the commutation behavior present
in the previous example relative to the skew bicharacter $\eps$, we introduce the concept
of $\eps$-color commutativity.

\begin{definition} \label{def:colorComm}
Let $A$ be a $G$-graded $\kk$-algebra with decomposition
$A = \bigoplus_{\sigma \in G} A_\sigma$, and let $\eps$ be a skew bicharacter
defined on $G$.  We say that $A$ is \emph{$\eps$-color commutative} (or simply
\emph{color commutative} if $\eps$ is understood) if for every
$x \in A_\sigma$ and $y \in A_\tau$, one has $xy = \eps(\sigma,\tau)yx$.
An element $x \in A_\sigma$ is said to be $G$-\emph{homogeneous}.
We call the $G$-degree of a $G$-homogeneous element $x$ the \emph{color} of $x$,
and we denote this by $\gdeg{x}$.  If $x$ and $y$ are $G$-homogeneous we abuse 
notation and use $\eps(x,y)$ to denote $\eps(\gdeg{x},\gdeg{y})$.
\end{definition}

We regard $Q$ as a color commutative ring corresponding
to the $G$-grading and associated skew bicharacter defined as in Example 
\ref{ex:skewColor}.  


\begin{notation}
Let $n \geq 1$ be an integer.  For each $\bfI = (i_1,\dots,i_n) \in \bbN^n$, we denote the product 
$x_1^{i_1}\cdots x_n^{i_n}$ by $x^\bfI$, which is a $\kk$-basis of $Q$.
Given an element $f \in Q$, we let $\supp(f)$
denote the set of multiindices $\bfI$ such that the coefficient of $x^\bfI$ is nonzero
in the unique expression of $f$ as a linear combination of monomials.
\end{notation}

The following lemma regarding normal elements of $Q$ generalizes 
\cite[Lemma 3.5]{KKZ}, and characterizes the components of the
$G$-grading of $Q$ in terms of its normal elements.

\begin{lemma} \label{normalLemma} Let $Q$ be a skew polynomial ring and let $\eps$ be its associated skew bicharacter.  Then:
\begin{enumerate}
\item \label{enum:normalLemma1} A homogeneous element $f \in Q$ is normal with
normalizing automorphism $\sigma$ if and only if $f$ is $G$-homogeneous of
color $\sigma$.
\item If $f,g\in Q$ are normal and homogeneous, then there exists $p\in \kk$ such that $fg=pgf$.
\item \label{enum:normalLemma3} If $f \in Q$ is normal then $\eps(x^\bfI,x^\bfJ) = 1$
for all $\bfI,\bfJ\in \supp(f)$.
\end{enumerate}
\end{lemma}
\begin{proof}
If a homogeneous element $f$ of $Q$ is $G$-homogeneous, then $f$ is normal,
since it skew commutes with the variables.
For the converse, suppose that $f = \sum_{\bfI \in \bbN^n} c_\bfI x^\bfI.$  It suffices to show that
$\eps(x^\bfI,x_j) = \eps(x^{\bfI'},x_j)$ for all $\bfI, \bfI' \in \supp(f),$
since then the $G$-degree of each monomial in the support of $f$ will be the same.

This claim is easily verified if $|\supp(f)|=1$, therefore we assume $|\supp(f)|\geq2$. If 
there is an $x^\bfK$ such that $x^\bfK$ divides $x^\bfI$ for all $\bfI\in \supp(f)$ then we can write
$f$ as $x^\bfK g$ with $g$ normal. Indeed, if $h\in Q$ then $fh=\tau(h)f$ where $\tau$ is the 
normalizing automorphism of $f$.  Thus $x^\bfK gh=\tau(h)x^\bfK g$, but since $x^\bfK$ is normal we 
can write $x^\bfK gh=x^\bfK \tau'(h)g$, hence $x^\bfK (gh-\tau'(h)g)=0$.  Since $x^\bfK$ is regular
we deduce $gh=\tau'(h)g$, i.e. $g$ is normal.
If the claim is true for $g$ then it is true for $f$, therefore we can assume that there is no 
$x^\bfK$ dividing $x^\bfI$ for all $\bfI\in \supp(f)$.  

In what follows, we use $\overline{(~)}$ to denote the image of an element in $Q/(x_j)$;
notice that $\bar f$ is nonzero for any choice of $x_j$. Since $f$ is normal we have 
\[
fx_j=(\sum_k\beta_k x_k)f,
\]
which modulo $x_j$ gives 
\[
0=(\sum_{k\neq j}\beta_k \overline{x}_k)\overline {f} .
\]
Hence $\beta_k=0$ for all $k\neq j$ and
\[
fx_j=\beta_j x_j f.
\]
It follows that $\eps(x^{\bfI},x_j) = \eps(x^{\bfI'},x_j) = \beta_j$ for all $\bfI,\bfI'\in \supp(f)$.

The second claim follows since by \eqref{enum:normalLemma1}, $f$ and $g$ are $G$-homogeneous,
therefore $fg = \eps(f,g)gf$.  The final claim follows since $\eps(\sigma,\sigma) = 1$ for all $\sigma \in G$.
\end{proof}

The following proposition shows that color commutative algebras are not exotic.
Our motivation for their introduction is to contain the proliferation of constants
that arise in computations with skew commutative polynomial rings.

\begin{proposition} \label{prop:colorCommSkew}
Let $A$ be a finitely generated connected graded $\kk$-algebra.  
Then $A$ is color commutative if and only if it is a quotient of a skew polynomial 
ring by an ideal generated by homogeneous normal elements.
\end{proposition}

\begin{proof}
Suppose $A$ is color commutative and generated by $\{x_1,\dots,x_n\}$.
Since $A$ is $G$-graded, we can decompose each $x_i$ according to the $G$-grading
and hence may assume this generating set is $G$-homogeneous.  Since $A$ is color
commutative, we have that $x_ix_j = \eps(x_i,x_j)x_jx_i$ so that $A$ is a quotient
of a skew polynomial ring $Q$.  Since the projection from $Q$ to $A$
is $G$-homogeneous, its kernel is $G$-homogeneous and hence is generated
by normal elements.  The converse follows from Example \ref{ex:skewColor}
and Lemma \ref{normalLemma}\eqref{enum:normalLemma1}.
\end{proof}

\section{Color DG Algebras} \label{sec:CDGA}

In this section, we introduce the main tool we use to study homological
properties of color commutative algebras.  It is a natural extension of the
theory of DG algebras over a commutative ring.  

\begin{definition} \label{def:colorDG}
Let $R$ be a color commutative connected graded $\kk$-algebra. Let $A$ be a DG $R$-algebra as in Definition
\ref{def:dga}.  We say that $A$ is a \emph{$\eps$-color DG $R$-algebra}
provided $A$ is $G$-graded with a grading compatible with the bigrading of $A$,
and the differential on $A$ is also $G$-homogeneous of color $e_G$.  We similarly define
the notion of left (or right) \emph{$\eps$-color DG $A$-module}.
If the bicharacter $\eps$ is understood, we simply call the above
notions \emph{color DG algebras/modules}.

We say that an element of $A$ is \emph{trihomogeneous} if it is bihomogeneous and $G$-homogeneous.

We also assume that a color DG $R$-algebra $A$ is \emph{graded color commutative}.  
That is, for all trihomogeneous 
$x,y \in A$, we assume that $xy = (-1)^{|x||y|}\eps(x,y)yx$, and that $x^2 = 0$ when
$x$ is trihomogeneous of odd homological degree.  As in the commutative case, the first
condition implies the second when the characteristic of $\kk$ is not $2$.
\end{definition}

It follows that if $A$ is a color DG algebra, then $Z(A)$ and $B(A)$ are
$G$-graded, and hence $H(A)$ carries the structure of a $G$-graded algebra.
Similarly, if $U$ is a color DG $A$-module, one also has that $H(U)$ is a color left $H(A)$ module.
An important (but simple) observation is that $Z(A)$ (and hence $H(A)$) is
color commutative so that $G$-homogeneous elements of $Z(A)$ skew commute with one another.
This is essentially the reason for working in the color commutative setting,
since now the hypotheses on the cycles appearing in Proposition \ref{prop:stillNormal} are
automatically satisfied.

Note that while the skew bicharacter $\eps$ does not enter the definition of a
$\eps$-color DG algebra, it does play a role in the notion of an $A$-linear 
map; see Definition \ref{def:colorLin} below.  When $x$ and $y$ are $G$-homogeneous
of $G$-degree $\sigma$ and $\tau$ respectively, we continue to abuse notation 
and write $\eps(x,y)$ for $\eps(\sigma,\tau)$.

\begin{remark} \label{rem:threeGradings2} 
The notion of color DG algebra brings a third grading into the mix: an internal,
a homological, and a $G$-grading.  For elements of a color DG module we will use the same terminology that has been introduced in Definition \ref{def:bihom}, \ref{def:colorComm} and \ref{def:colorDG}. Given a trihomogeneous element $u$ in a color DG module $U$, we denote 
the internal, homological and group degree of $u$ by $\deg(u)$, $|u|$ and
$\gdeg{u}$, respectively.

We denote the component of $U$ in homological degree $m$, group degree $\sigma$,
and internal degree $n$ by $U_{m,\sigma,n}$.  These indices are listed in order of
relevance for our computations, and so we also adopt the convention that when
fewer indices are used, they are left off of the end.  That is, we use $U_{m,g}$ to denote the 
component of $U$ in homological degree $m$ and group degree $g$, and
$U_m$ to denote the component of $U$ in homological degree $m$.

Let $U$ and $V$ be color DG modules over the color DG algebra $A$. A homomorphism of color DG modules
$\varphi: U\rightarrow V$ is said to be \emph{homogeneous} of (internal) degree $n\in\mathbb{Z}$ if
$\deg \varphi(u)=\deg u+n$ for all homogeneous elements $u\in U$, in which case we write $\deg\varphi=n$.
It is said to be \emph{bihomogeneous} of homological degree $m\in\mathbb{Z}$ if it is homogeneous and 
$|\varphi(u)|=|u|+m$ for all bihomogeneous elements $u\in U$, in which case we write $|\varphi|=m$. It
is said to be \emph{trihomogeneous} of color $\sigma$ if it is bihomogeneous and 
$\gdeg{\varphi(u)}=\sigma\gdeg{u}$ for all trihomogeneous elements $u\in U$, in which case we write 
$\gdeg{\varphi}=\sigma$. The map $\varphi$ is said to be $G$-homogeneous of color $\sigma$ if 
$\gdeg{\varphi(u)}=\sigma\gdeg{u}$ for all $G$-homogeneous elements $u\in U$, in which case we write 
$\gdeg{\varphi}=\sigma$.
\end{remark}

Before continuing with some basic remarks on color DG modules, we record some
facts regarding our constructions in Section \ref{sec:killCycle}. 

\begin{proposition} \label{prop:colorDGAy}
Suppose that $A$ is a color DG algebra, and $z \in Z(A)$ is a 
trihomogeneous cycle.  Then $\Ay{A}{y~|~\del(y) = z}$ is a color DG algebra.
In particular, if $R$ is a color commutative connected graded $\kk$-algebra,
and $f_1,\dots,f_c$ is a sequence of normal elements of $R$, then the Koszul complex
$K^R(\mathbf{f})$ given in Construction \ref{Kos} is a color DG algebra.
\end{proposition}

\begin{proof}
By assigning the internal and $G$-degree of $y$ to that of $z$, the differential
of $\Ay{A}{y}$ is homogeneous and $G$-homogeneous.  Showing that $\Ay{A}{y}$ is
color commutative is a straightforward verification using Lemma
\ref{normalLemma}\eqref{enum:normalLemma3} and the
fact that the normalizing automorphism of $y$ is the same as that of $z$.

It remains to prove that trihomogeneous elements of $\Ay{A}{y}$ of odd homological degree square to zero.
We prove the case that $y$ has even degree, the odd degree case is similar.  Suppose that
$z = \sum a_i y^{(i)}$ is trihomogeneous of odd degree.  Then
\begin{eqnarray*}
z^2 & = & \sum_{i,j} (a_iy^{(i)})(a_jy^{(j)}) \\
    & = & \sum_{i < j} \left((a_iy^{(i)})(a_jy^{(j)}) + (a_jy^{(j)})(a_iy^{(i)})\right) + \sum_i (a_iy^{(i)})^2.
\end{eqnarray*}
Each term in the first sum in the last line of the display is zero since for all $i$,
$a_iy^{(i)}$ has the same $G$-degree as $z$, $\eps(z,z) = 1$, and $z$ has odd homological degree.
Each term in the second sum is zero since $a_i^2 = 0$ for each $i$.
\end{proof}

\begin{construction} \label{constr:AcyClos}
Let $R$ be a color commutative connected graded $\kk$-algebra.  Let $F_1$ be the Koszul complex as constructed
in Construction \ref{Kos}.   For $n \geq 1$, set
$$F_{n+1}=\Ay{F_n}{y_1,\ldots,y_m\mid \partial(y_i)=z_i,\;i=1,\ldots,m}$$
where $\HH{n}{F_n}=\langle\cls{z_1},\ldots,\cls{z_m}\rangle$.  This iterative process is possible
since the set of representative cycles chosen as generators must skew commute because they
are $G$-homogeneous.

Then the complex $F=\varinjlim F_n$ is a free resolution of $\kk$ over $R$. In Section \ref{sec:AcyClos} we will prove that if the generators of $\HH{n}{F_n}$ are chosen minimally for all $n$ then the resolution $F$ will be minimal.
\end{construction}

Next, we develop some basic properties of color DG modules over a color DG algebra.
This involves the next construction, which helps control the introduction
of some constants in the development that follows.

\begin{definition} \label{def:rightAction}
Let $A$ be a color DG algebra.  Let $U$ be a left $A$-module.
Then $U$ may be given the structure of an $A$-bimodule by setting
$ua = (-1)^{|u||a|}\eps(\sigma,\tau)au$ for all $u$ of $G$-degree $\sigma$
and $a$ of $G$-degree $\tau$.  Note that if $A \to B$ is a homomorphism of color DG algebras,
and $B$ is viewed as a left $A$-module, then this bimodule structure is
compatible with the usual one.
\end{definition}

\begin{definition} \label{def:colorLin} Let $U$ and $V$ be left color DG modules
over a color DG algebra $A$.  A homomorphism of complexes $\beta : U \to V$
is said to be \emph{(left) $A$-linear of color $\sigma$} if it is trihomogeneous of
group degree $\sigma$, and if for all trihomogeneous elements $a \in A$ and $u \in U$, one 
has $\beta(au) = (-1)^{|a||\beta|}\eps(\beta,a)a\beta(u)$.
The subspace of $\Hom_\kk(U,V)$ spanned by all $A$-linear trihomogeneous homomorphisms of complexes is denoted $\Hom_A(U,V)$.  By definition, this set is
$G$-graded, with $\sigma$ component given by the set of all $A$-linear 
homomorphisms of color $\sigma$ from $U$ to $V$. By definition, this set also has a natural internal and homological grading.

For all trihomogeneous $a,\beta,u$ as above, the action
$$(a\beta)(u) = a(\beta(u)) = (-1)^{|a||\beta|}\eps(a,\beta)\beta(au)$$
and differential $\del(\beta) = \del^V\beta - (-1)^{|\beta|}\beta\del^U$
gives $\Hom_A(U,V)$ the structure of a color DG module over $A$.
As usual, there is a correspondence between cycles of $\Hom_A(U,V)$
and $A$-linear chain maps from $U$ to $V$, as well as a correspondence
between homotopy classes of $A$-linear chain maps from $U$ to $V$ and homology 
classes in $\Hom_A(U,V)$.
\end{definition}

\begin{remark} \label{rem:rightLinear}
Let $A$ be a color DG algebra with $A_0 = R$ a color commutative $\kk$-algebra.
Note that while an $A$-linear map $\psi \in \Hom_A(U,V)$ is not left $R$-linear,
it is right $R$-linear using the right action introduced in Definition \ref{def:rightAction}.
Indeed, for all trihomogeneous elements $u \in U$ and homogeneous and $G$-homogeneous elements $r \in R$, one has
\begin{align*}
\psi(ur)&=\psi(\e(u,r)ru)\\
&=\e(u,r)\psi(ru)\\
&=\e(u,r)\e(\psi,r)r\psi(u)\\
&=\e(\psi(u),r)r\psi(u)\\
&=\psi(u)r.
\end{align*}
It follows that we may view elements of $\Hom_A(U,V)$ (which are defined using
a left color linearity condition) as \emph{right} linear homomorphisms of $R$-modules.
\end{remark}

\begin{definition} Let $U$ and $V$ be color DG modules over $A$.  The tensor
product $U \otimes_A V$ is the quotient of the (trigraded) tensor product
$U \otimes_\kk V$ by the vector space spanned by all elements of the form
$ua \otimes_\kk v - u \otimes_\kk av$.  The left action of $A$ on $U$
provides the $A$-action on $U \otimes_A V$, and the differential
$$\del(u \otimes v) = \del(u)\otimes v + (-1)^{|u|}u\otimes\del v$$
gives $U \otimes_A V$ the structure of a color DG module over $A$.
\end{definition}

\begin{proposition} \label{prop:adjoint}
Let $A$ and $B$ be color DG algebras, and $A \to B$ a morphism of DG algebras.
Let $U$ be a color DG $A$-module, and $V,W$ color DG $B$-modules.  Then the 
map
\begin{eqnarray*}
\Phi : \Hom_A(U, \Hom_B(V,W)) & \to & \Hom_B(U \otimes_A V, W) \\
\varphi & \mapsto & \left( u \otimes v \mapsto \varphi(u)(v) \right),
\end{eqnarray*}
is an isomorphism of color DG $A$-modules.
\end{proposition}

\begin{proof}
The map $\Phi(\varphi)$ is clearly $\kk$-linear and trihomogeneous.
To see that it preserves
the relations of $U \otimes_A V$, note that if $\varphi$ is trihomogeneous,
one has
\begin{align*}
\Phi(\varphi)(ua \otimes v) &= \varphi(ua)(v) \\
         &= (-1)^{|a||u|}\eps(u,a)\varphi(au)(v) \\
         &= (-1)^{|a|(|u| + |\varphi|)}\eps(u,a)\eps(\varphi,a)a\varphi(u)(v) \\
         &= \varphi(u)(av) = \Phi(\varphi)(u\otimes av). 
\end{align*}
The remaining claims are easily checked.
\end{proof}

To finish the section, we record some properties of an important class of color DG modules - 
those whose underlying module is is free.

\begin{definition} A bounded below color DG module $F$ over $A$ is \emph{semi-free} if its
underlying $A^\natural$-module $F^\natural$ has an $A^\natural$-basis $\{e_\lambda\}_{\lambda \in
\Lambda}$ which is trihomogeneous.
\end{definition}

The following proposition appears as \cite[Proposition 1.3.1]{IFR} in the case of
DG algebras.  The same proof that is given there is applicable here as well.
\begin{proposition} \label{prop:semiFreeLift}
Suppose $F$ is a semi-free color DG module over a color DG algebra $A$.
Then each diagram of morphisms of color DG modules over $A$ represented by solid
arrows
$$\xymatrix{
   & U \ar[d]^\beta_\simeq \\
F \ar@{.>}[ur]^\gamma \ar[r]_\alpha & V
}$$
with a surjective quasi-isomorphism $\beta$ can be completed to a commutative diagram
by a morphism $\gamma$ that is uniquely defined up to $A$-linear homotopy.
\end{proposition}

Standard arguments (c.f. \cite[Proposition 1.3.2]{IFR}) also provide the following proposition.

\begin{proposition} \label{prop:homPreservesQuism}
If $F$ is a semi-free color DG module, then each quasi-isomorphism
$\beta : U \to V$ of color DG modules over $A$ induces quasi-isomorphisms
$$\Hom_A(F,\beta) : \Hom_A(F,U) \to \Hom_A(F,V); \qquad
  F \otimes_A \beta : F \otimes_A U \to F \otimes_A V.$$
\end{proposition}

\section{The acyclic closure and its properties} \label{sec:AcyClos}

In this section, we explore properties of semi-free extensions of the form
$\Ay{A}{Y}$ where $A$ is a color DG algebra, and $Y$ is a set of divided
powers variables whose differential is compatible with all gradings; see Constructions
\ref{cons:adjVarOdd} and \ref{cons:adjVarEven}.

\begin{definition}
Let $A$ be a color DG algebra, and let $U$ be a color module
over $\Ay{A}{Y}^\natural$, where $Y$ is a set of exterior and divided power variables
of positive degree.  A trihomogeneous $\kk$-linear map $D : \Ay{A}{Y} \to U$
such that
\begin{align*}
D(a)   =~& 0 & & \text{for all}~a \in A, \\
D(bb') =~& D(b)b' + (-1)^{|D||b|}\eps(D,b)bD(b') & & \text{for all}~b,b' \in \Ay{A}{Y}, b,b'~\text{trihomogeneous}, \\
D(y^{(i)}) =~& D(y)y^{(i-1)} & & \text{for all}~y \in Y_\text{even}~\text{and all}~i \in \bbN,
\end{align*}
is called an \emph{$A$-linear color derivation}.  Note that according to Definition \ref{def:colorLin}
such a map is a homomorphism of left color $A^\natural$-modules.
\end{definition}

Using the above properties, one sees that
$$D(y_{\lambda_1}^{(i_{\lambda_1})}\cdots y_{\lambda_q}^{(i_{\lambda_q})}) = 
\sum_{j=1}^q (-1)^{s_{j-1}} c_{j-1} y_{\lambda_1}^{(i_{\lambda_1})}\cdots
D(y_{\lambda_j}^{(i_{\lambda_j})})\cdots y_{\lambda_q}^{(i_{\lambda_q})}$$
where $s_j = |D|(i_{\lambda_1}|x_{\lambda_1}| + \cdots + i_{\lambda_j}|x_{\lambda_j}|)$,
and $c_j = \eps(D,y_{\lambda_1}^{(i_{\lambda_1})}\cdots y_{\lambda_j}^{(i_{\lambda_j})})$.
This implies that a derivation $D$ is determined by its value on $Y$.
In fact, the converse is true:  each trihomogeneous function $Y \to U$ extends to a unique
$A$-linear color derivation from $\Ay{A}{Y}$ to $U$ using the formulas above and $A$-linearity.

For a fixed $\sigma \in G$, we denote the set of all $A$-linear derivations
of degree $\sigma$ and homological degree $i$ from $\Ay{A}{Y}$ to $U$ by $\Der_A(\Ay{A}{Y},U)_{i,\sigma}$, which is
a left $A$-module using the action given in Definition \ref{def:colorLin}.  
Finally, we let $\Der_A(\Ay{A}{Y},U)$ denote the direct sum
$\bigoplus_{(i,\sigma) \in \bbZ \times G} \Der_A(\Ay{A}{Y},U)_{i,\sigma}$.

One may check that if $U$ is a color DG module over $\Ay{A}{Y}$, then 
$\Der_A(\Ay{A}{Y},U)$ is a color DG $\Ay{A}{Y}$-submodule of $\Hom_A(\Ay{A}{Y},U)$.
Further, if $\beta : U \to V$ is a homomorphism of color DG modules over $\Ay{A}{Y}$,
then the induced map
$$\Der_A(\Ay{A}{Y},\beta) : \Der_A(\Ay{A}{Y},U) \to \Der_A(\Ay{A}{Y},V)$$
is also a homomorphism of color DG modules over $\Ay{A}{Y}$; that is,
$\Der_A(\Ay{A}{Y},-)$ is an endofunctor of the category of color DG modules
over $\Ay{A}{Y}$.  The next proposition shows that this functor is
representable; its proof is a straightforward adaptation
of \cite[Proposition 6.2.3]{IFR} and is omitted.


\begin{proposition} \label{prop:Diff}
Let $\Ay{A}{Y}$ be a semi-free extension of the color DG algebra $A$.  
Then there exists a semi-free color DG module $\Diff_A\Ay{A}{Y}$ over $\Ay{A}{Y}$
and a degree $(0,e_G)$ chain derivation $d : \Ay{A}{Y} \to \Diff_A\Ay{A}{Y}$
such that
\begin{enumerate}
\item The $\Ay{A}{Y}^\natural$-module $(\Diff_A(\Ay{A}{Y})^\natural$ has a basis
$dY = \{dy~|~y \in Y\}$, where $dy$ and $y$ have the same internal, homological,
and group gradings.
\item $d(y) = dy$ for all $y \in Y$.
\item $\del(b(dy)) = \del(b(dy)) + (-1)^{|b|}bd(\del(y))$ for all
$b \in \Ay{A}{Y}$.
\item The map 
\begin{eqnarray*}
\Hom_{\Ay{A}{Y}}(\Diff_A \Ay{A}{Y},U) & \to & \Der_A(\Ay{A}{y},U) \\
 \beta & \mapsto & \beta \circ d
\end{eqnarray*}
is an isomorphism (natural in $U$) of color DG modules over $\Ay{A}{Y}$
with inverse given by
$$\vartheta \mapsto \left(\sum_{y \in Y} a_ydy \mapsto 
  \sum_{y \in Y}(-1)^{|\vartheta||a_y|}\eps(\vartheta,a_y)a_y\vartheta(y)\right).$$
\end{enumerate}
\end{proposition}
Using Propositions \ref{prop:semiFreeLift} and \ref{prop:homPreservesQuism}
one therefore obtains the following corollary.
\begin{corollary} \label{cor:derQuism}
If $U \to V$ is a (surjective) quasi-isomorphism of color DG modules over
$\Ay{A}{Y}$, then so is the induced map
$\Der_A(\Ay{A}{Y},U) \to \Der_A(\Ay{A}{Y},V)$.
\end{corollary}

We next wish to construct derivations corresponding to the variables $Y$
added in the extension $\Ay{A}{Y}$.  We continue to follow the treatment
in \cite[Section 6]{IFR}.

\begin{construction} \label{const:indec}
Let $\Ay{A}{Y}$ be a semi-free extension of $A$ and fix a total order of
the variables $Y$.  Let $J$ denote the kernel of the morphism $A \to B = H_0(\Ay{A}{Y})$.  Note that
we consider $B$ as a color DG algebra with trivial differential.  Let
$Y^{(\geq 2)}$ be the set of normal monomials that are decomposable;
that is, the variables appearing in the monomial are sorted in
increasing order, and their word length in the $Y$ variables is two or more.
Then $A + JY + AY^{(\geq 2)}$ is a DG $A$-submodule of $\Ay{A}{Y}$.

The \emph{complex of indecomposables of the extension $\Ay{A}{Y}$} is defined
to be the quotient complex $\Ay{A}{Y}/(A + JY + AY^{(\geq 2)})$, and it is
denoted $\Ind_A(\Ay{A}{Y})$.  It is a DG module over $\Ay{A}{Y}$ which
is also a complex of free $B$-modules with basis $Y_n$ in homological
degree $n$, where $Y_n$ denotes the elements of $Y$ of homological degree $n$.
We denote by $\pi$ the projection $\Ay{A}{Y} \to \Ind_A(\Ay{A}{Y})$.
\end{construction}

One may use the complex of indecomposables to define derivations
using the following lemma.

\begin{lemma} \label{lem:derivExist}
Let $V$ be a color DG module over $B$, let $U$ be a color DG module
over $A$ with $U_i = 0$ for $i < 0$, and suppose that
$\beta : U \to V$ is a surjective quasi-isomorphism.

For each $B$-linear map $\xi : \Ind_A(\Ay{A}{Y}) \to V$
of degree $(-n,\sigma)$, there exists
an $A$-linear chain derivation $\vartheta : \Ay{A}{Y} \to U$ of the same
degree such that $\beta\vartheta = \xi\pi$, and any two such derivations
are homotopic by a homotopy that is itself an $A$-linear derivation.

Furthermore, if $\{u_y\} \subseteq U_0$ is a collection of elements such that 
$\beta(u_y) = \xi(y)$ for all $y \in Y_n$ then there is a chain derivation 
$\vartheta$ satisfying $\vartheta(y) = u_y$ for all $y \in Y_n$.
\end{lemma}

\begin{proof}
Set $D = \Diff_A \Ay{A}{Y}$.  Since the projection $\pi$ is an
$A$-linear chain derivation, by Proposition \ref{prop:Diff} there is
an induced $\Ay{A}{Y}$-linear morphism $D \to \Ind_A(\Ay{A}{Y})$
such that $\pi(dy) = \pi(y) = y$ for all $y \in Y$.  This in turn induces
a morphism of $B$-complexes
$B \otimes_{\Ay{A}{Y}} D \to \Ind_A(\Ay{A}{Y})$.
This map is bijective on a basis of each and is hence an isomorphism.

Combining this map with the adjoint isomorphism in Proposition 
\ref{prop:adjoint}, the universal property of $D$, and Corollary 
\ref{cor:derQuism} gives a surjective quasi-isomorphism
$$\Der_A(\Ay{A}{Y},U) \xra{\simeq} \Hom_B(\Ind_A(\Ay{A}{Y}),V).$$
Therefore given $\xi$ as in our hypothesis, there exists a chain derivation
$\vartheta : \Ay{A}{Y} \to U$ that satisfies $\beta\vartheta = \xi\pi$
as claimed.  Any two choices must necessarily differ by a boundary of
$\Der_A(\Ay{A}{Y},U)$, i.e.\ a homotopy which itself is an $A$-linear derivation.

For the last claim, suppose that $\{u_y\} \subseteq U_0$ is as in the statement of the lemma,
i.e. we have chosen a lifting $u_y$ of each $\xi(y)$ along $\beta$.  Since
$U_i = 0$ for $i < 0$, we are free to choose $\vartheta_n$ to be any map
that satisfies $(\beta_0\vartheta_n)(y) = (\xi_n\pi_n)(y)$ for all $y \in Y_n$,
with setting $\vartheta_n(y) = u_y$ one such choice.
\end{proof}

We finish this section by showing that if the cycles in the construction
of $\Ay{A}{Y}$ are chosen minimally (in a sense made more precise in the following definition),
then the underlying complex is minimal.

\begin{definition} \label{def:acyclicClosure}
Let $A$ be a nonnegatively graded color DG $\kk$-algebra such that $R = A_0$ is a quotient of
a skew polynomial ring by a sequence of homogeneous normal elements, and suppose each right
$R$-module $H_n(A)$ is finitely generated. Let $A \to B$ be a surjective map of color
DG algebras with $B$ concentrated in degree zero (i.e. $B$ is a color commutative 
$\kk$-algebra) with $J = \ker(A \to B)$ as before, such that $J_0$ is generated by a sequence of normal elements.

Construction \ref{constr:AcyClos} can be applied to produce a resolution $\Ay{A}{Y}$ of $B$ over $A$. If $\Ay{A}{Y}$ satisfies the following two conditions:
\begin{enumerate}
\item $\del(Y_1)$ minimally generates $J_0$ modulo $\del_1(A_1)$, and
\item $\{\cls{\del(y)}~|~y\in Y_{n+1}\}$ minimally generates $H_n(\Ay{A}{Y_{\leq n}})$ for $n \geq 1$,
\end{enumerate}
then $\Ay{A}{Y}$ is called an \emph{acyclic closure of $B$ over $A$.}
\end{definition}

When naming the variables we adjoin in an acyclic closure, we do so using the natural
numbers such that $|y_i| \leq |y_j|$ when $i < j$.  It follows that we may totally
order the normal monomials of $Y$ using the graded lexicographic order induced by this
total order of $Y$.

We will denote by $y^{(I)}$ the normal monomial $y_1^{(i_1)}\cdots y_q^{(i_q)}$
when $I = (i_1,\dots,i_q)$ is a finite indexing sequence.  Any finite indexing sequence
may be padded at the end with zeroes in order to make comparisons using graded
lexicographic order.  The following lemma appears in \cite[Lemma 6.3.3]{IFR}.
Its proof is the same as the one given in \emph{loc.\ cit}, except one uses Lemma
\ref{lem:derivExist} above, which adapted the commutative setting to the color commutative one.

\begin{lemma} \label{lem:UpTriang}
Let $\Ay{A}{Y}$ be an acyclic closure of $\kk$ over $A$.  Then there exist $A$-linear
chain derivations $\vartheta_i : \Ay{A}{Y} \to \Ay{A}{Y}$ for $i \geq 1$ such that:
\begin{enumerate}
\item $$\vartheta_i(y_h) = \begin{cases} 0 &\text{for}~|y_h| \leq |y_i|~\text{and}~h \neq i\\
                                    1 &\text{for}~h = i.\end{cases}$$
\item \label{lem:UpTriang2}
For a finite indexing sequence $I = (i_1,\dots,i_q)$, let $\vartheta^I$
denote the map $\vartheta_q^{i_q}\cdots\vartheta_1^{i_1}$, and let $H \leq I$ be another
indexing sequence.  Then one has
$$\vartheta^I(y^{(H)}) = \begin{cases} 0 &\text{for}~H < I \\
                                       1 &\text{for}~H = I.\end{cases}$$
\item Each $\vartheta_i$ is unique up to an $A$-linear homotopy which is a color
derivation.
\end{enumerate}
\end{lemma}

We now obtain the following result which was first proved by Gulliksen \cite{GulLev} and
Schoeller \cite{Scho} in the commutative case.

\begin{theorem}
Let $A$ be a color DG algebra such that $A_0 = R$ is a quotient of a skew polynomial
ring over the field $\kk$ by a sequence of homogeneous normal elements,
and suppose each right $R$-module $H_n(A)$ is finitely generated.  If
$\Ay{A}{Y}$ is an acyclic closure of $\kk$, then $\del(\Ay{A}{Y}) \subseteq J\Ay{A}{Y}$
where $J = \ker(A \to \kk)$.

In particular, if $R$ is as above, and if $\Ay{R}{Y}$ is an acyclic closure of $\kk$
over $R$, then $\Ay{R}{Y}$ is a minimal free resolution of $\kk$ over $R$.
\end{theorem}

\begin{proof}
Let $b$ be a trihomogeneous element of $\Ay{A}{Y}$, and write
$\del(b) = \sum_H a_H y^{(H)}$, where each $y^{(H)}$ is a normal monomial.
Proving that $\del(\Ay{A}{Y})$ is a subset of $J\Ay{A}{Y}$ is
equivalent to showing that if $|y^{(H)}| = |b| - 1$ and $a_H \neq 0$, then $\deg(a_H) > 0$.  Suppose
that there exists an index $I$ such that $|y^{(I)}| = |b| - 1$, $a_I \neq 0$,
$\deg(a_I) = 0$, and that $I$ is chosen to be least in the graded lexicographic order
with respect to this property.

Applying the chain derivation $\vartheta^I$ from Lemma \ref{lem:UpTriang} to 
$\del(b)$ we obtain
$$\pm \del_1(\vartheta^I(b)) = \vartheta^I(\del(b)) = a_I + \sum_{H > I}a_H\vartheta^I(y^{(H)}) \equiv a_I \mod R_+\Ay{A}{Y}$$
where the second equality follows from
Lemma \ref{lem:UpTriang}\eqref{lem:UpTriang2}.
Therefore $\del_1(\vartheta^I(b)) \not\in R_+.$  However, since $A_0 = R$ and 
$\Ay{A}{Y}$ is an acyclic closure of $\kk$, we must have that
$\del_1(\Ay{A}{Y}) = R_+$, giving us our contradiction.
\end{proof}

\section{The category of color commutative DG algebras with divided powers}
\label{sec:CDGADP}
In this section we prove the uniqueness of acyclic closures in the category 
of color commutative DG algebras with divided powers.
We follow the development that appears in \cite{GulLev} for the case
of commutative rings.
\begin{definition}
Let $A$ be a color DG $R$-algebra with respect to a skew bicharacter $\e : G \times G \to \kk^*$,
and with differential $\partial$. We say that $A$ is a
\emph{color DG algebra with divided powers} if to every trihomogeneous
element $x\in A$ of even positive homological degree there is associated a sequence of elements 
$x^{(k)}\in A$ ($k=0,1,2,\ldots$) satisfying:
\begin{enumerate}
\item $x^{(0)}=1, x^{(1)}=x, |x^{(k)}|=k|x|, \gdeg{x^{(k)}}=(\gdeg{x})^k,$
\item $x^{(h)}x^{(k)}=\binom{h+k}{h}x^{(h+k)},$
\item if $(|x|,\gdeg{x})=(|y|,\gdeg{y})$ then
\[
(x+y)^{(k)}=\sum_{i+j=k}x^{(i)}y^{(j)},
\]
\item for $k\geq 2$ and $y$ trihomogeneous
\[
(xy)^{(k)}=
\begin{cases}
0&\mathrm{if}\;|x|\;\mathrm{and}\;|y|\;\mathrm{are\;odd}\\
\e(y,x)^{\binom{k}{2}}x^ky^{(k)}&\mathrm{if}\;|x|\;\mathrm{is\;even\;and}\;|y|\;\mathrm{is\;even\;and\;positive},
\end{cases}
\]
where $\binom{k}{2}=0$ if $k<2$,
\item $$(x^{(h)})^{(k)}=\bbinom{h}{k}x^{(hk)},\quad\mathrm{where}\;\bbinom{h}{k}=\frac{(hk)!}{k!(h!)^k}$$
\item for $k\geq1$
\[
\partial(x^{(k)})=(\partial (x))x^{(k-1)}.
\]
\end{enumerate}
The scalars $\binom{h}{k}$ and $\bbinom{h}{k}$ are computed in $\mathbb{Z}$ and then reduced modulo the characteristic of $\kk$. 
\end{definition}
It remains to notice that $\bbinom{h}{k}$ is an integer, this follows since $\bbinom{h}{0}=1$ and for $k\geq1$ from the recursive relation
\[
\bbinom{h}{k}=\bbinom{h}{k-1}\binom{(k-1)h}{h-1}
\]

\begin{definition}\label{defn:morph}
Let $A$ and $B$ be color DG $R$-algebras with divided powers, with the same skew bicharacter $\e$. A map 
$f:A\rightarrow B$ is a \emph{morphism of color DG algebras with divided powers} if it is a trihomogeneous 
morphism of color DG $R$-algebras, such that
\[
f(x^{(k)})=f(x)^{(k)},\quad\mathrm{for}\;k\geq0\;\mathrm{and}\;x\in A, |x|\;\mathrm{even\;and\;positive}.
\]
\end{definition}

\begin{definition}
The category $\cdgardp$ is the category with objects color DG algebras with divided powers with
skew bicharacter $\eps$, and morphisms as defined in Definition \ref{defn:morph}.
\end{definition}

\begin{lemma}\label{lem:extension}
Let
\[
\xymatrix{
\Ay{A}{Y} \ar@{.>}[r]^{\tilde{f}} & C\\
A \ar[u]^{g} \ar[r]^{f} & B \ar[u]_{h}\\
}
\]
be a diagram in $\cdgardp$ with $C$ acyclic, with $g$ and $h$ inclusions.
Then there exists a morphism $\tilde{f}:\Ay{A}{Y}\rightarrow C$ in $\cdgardp$ making the
diagram commutative.
\end{lemma}

\begin{proof}
By induction it suffices to prove the case $Y=\{y\}$, therefore $\Ay{A}{Y}=\Ay{A}{y\mid \partial^{\Ay{A}{y}}(y)=z}$. 
The acyclicity of $C$ implies that there exists $c\in C$ such that $\partial^C(c)=f(z)$. Since
$\partial^C, \partial^A$ and $f$ are $G$-homogeneous maps of degree $e_G$, it follows that
$\gdeg{c}=\gdeg{y}$.  If $|y|$ is even then $\Ay{A}{y}$ is a free $A$-module with basis $\{y^{(k)}\}_{k\geq0}$.
Define the map $\tilde{f}:\Ay{A}{y}\rightarrow C$ as: $\tilde{f}(ay^{(k)})=f(a)c^{(k)}$ for
$k\geq0, a\in A$ and then extend by left color $R$-linearity.
It is straightforward to check that $\tilde{f}$ is a map of algebras which
preserves divided powers, and that $\tilde{f}$ fits into the diagram
above.
\end{proof}

\begin{lemma} \label{lem:uniqueness} Let $R$ be a noetherian color commutative connected
$\kk$-algebra.  Let
$$\xymatrix{
\Ay{A}{Y} \ar[r]^{\tilde{f}} & \Ay{A'}{Y'} \\
A \ar[u]^{g} \ar[r]^{f} & A' \ar[u]_{h} \\
}
$$
be a commutative diagram in $\cdgardp$ of algebras with $g$ and $h$ free extensions
with $Y$ and $Y'$ concentrated in positive degree with finitely many variables
in each degree.  Assume further that $f$ an isomorphism, and $H_0(\Ay{A'}{Y'}) \neq 0$.
Then $\tilde{f}$ is an isomorphism if and only if 
$H_0(\tilde{f})$ is an isomorphism and the induced map
$\bar{f} : \Ind_A(\Ay{A}{Y}) \to \Ind_{A'}(\Ay{A'}{Y'})$ is an isomorphism.
\end{lemma}

\begin{proof} Set $B = H_0(\Ay{A}{Y})$ and $B' = H_0(\Ay{A'}{Y'})$.
Clearly if $\tilde{f}$ is an isomorphism then $H_0(\tilde{f})$ and
$\bar{f}$ are isomorphisms as well.  For the converse, as noted in 
Construction \ref{const:indec}, one has that $\Ind_A(\Ay{A}{Y})$ is a
complex of free $B$-modules, and since $B \cong B'$, we have that 
$\Ind_{A'}(\Ay{A'}{Y'})$ is as well.  Since $\bar{f}$ is an isomorphism,
it follows that for each $q$, $Y$ and $Y'$ contain the same number of 
variables in homological degree $q$.  Since $f$ is an isomorphism, this
implies that $\Ay{A}{Y}_q$ and $\Ay{A}{Y'}_q$ are isomorphic as well.
Since $R$ is noetherian and connected graded, in order to show that 
$\tilde{f}_q$ is an isomorphism, it suffices to show that $\tilde{f}_q$
is surjective.

We proceed by induction on $q$, with the case $q = 0$ being clear.  Assume
that $\tilde{f}_p$ is surjective for all $p < q$.  Since $\bar{f}$ is 
surjective, we have that
$$\Ay{A'}{Y'} \subseteq \tilde{f}(\Ay{A}{Y}_q) + A'_q + (J'Y')_q +
(A'Y'^{(\geq 2)})_q,$$
where $J' = \ker A' \to B'$ as in the definition of
$\Ind_{A'}(\Ay{A'}{Y'})$.  Since $f$ is surjective, we have
$A'_q \subseteq \tilde{f}(\Ay{A}{Y}_q)$.  Since $f$ preserves divided
powers, the induction hypothesis and surjectivity of $f$ implies
$(J'_{\geq 1}Y')_q$ and $(A'Y'^{(\geq 2)})_q$ are contained in $\tilde{f}
(\Ay{A}{Y}_q)$.   Since $H_0(\Ay{A'}{Y'})$ is nonzero, we know that $J'_0$ is 
contained in $R_+A'_0$, hence
$(J'_0Y') \subseteq R_+A'_0Y' \subseteq R_+\Ay{A'}{Y'}$.
Putting these facts together, we have that
$$\Ay{A'}{Y'}_q \subseteq \tilde{f}(\Ay{A}{Y}_q) + R_+\Ay{A'}{Y'}_q,$$
which gives surjectivity of $\tilde{f}_q$ by Nakayama's Lemma.
\end{proof}

We are finally ready to prove uniqueness of acyclic closures.  
\begin{theorem}\label{thm:uniqueness}
Let $\Ay{R}{Y}$ and $\Ay{R}{Y'}$ be acyclic closures of $\kk$ over $R$. Then there exists an isomorphism $\tilde{f}:\Ay{R}{Y}\rightarrow\Ay{R}{Y'}$ in the category $\cdgardp$ such that $f|_R=\id_R$.
\end{theorem}
\begin{proof}
The existence of $\tilde{f}$ follows from extending $\id_R:R\rightarrow R$ using Lemma \ref{lem:extension}. 
The map $\tilde{f}$ induces homomorphisms in $\cdgardp$ for all $i\geq0$
\[
\tilde{f}_i:\Ay{R}{Y_{\leq i}}\rightarrow \Ay{R}{Y^\prime_{\leq i}}.
\]
It suffices to prove by induction that $\tilde{f}_i$ is an isomorphism since then it will follow
from the exactness of direct limits that $\tilde{f}$ is also an isomorphism.

The map $\tilde{f}_0=\id_R$ is clearly an isomorphism. Let $i\geq1$ and consider the following commutative diagram with exact rows
$$\xymatrix{
0\ar[r]&\Ind_R(\Ay{R}{Y_{\leq i-1}})\ar[r]\ar[d]^{\alpha}&\Ind_R(\Ay{R}{Y_{\leq i}})\ar[r]\ar[d]^{\beta}&\Ind_{\Ay{R}{Y_{\leq i-1}}}(\Ay{R}{Y_{\leq i}})\ar[r]\ar[d]^{\gamma}&0\\
0\ar[r]&\Ind_R(\Ay{R}{Y^\prime_{\leq i-1}})\ar[r]&\Ind_R(\Ay{R}{Y'_{\leq i}})\ar[r]&\Ind_{\Ay{R}{Y'_{\leq i-1}}}(\Ay{R}{Y'_{\leq i}})\ar[r]&0\\
}
$$
where the vertical maps are induced by $\tilde{f}$.

Notice that by construction, the rows of the following diagram are isomorphisms of $\kk$-vector spaces induced by the differential of the acyclic closure 
$$\xymatrix{
\Ind_{\Ay{R}{Y_{\leq i-1}}}(\Ay{R}{Y_{\leq i}})\ar[r]\ar[d]^{\gamma}&\mathrm{H}_{i-1}(\Ay{R}{Y_{\leq i-1}})\otimes_R\kk\ar[d]^{\eta}\\
\Ind_{\Ay{R}{Y'_{\leq i-1}}}(\Ay{R}{Y'_{\leq i}})\ar[r]&\mathrm{H}_{i-1}(\Ay{R}{Y^\prime_{\leq i-1}})\otimes_R\kk.}
$$
It follows that if $\tilde{f}_{i-1}$ is an isomorphism then so is $\eta$ and hence $\gamma$. The map $\alpha$ is also an isomorphism, hence we deduce that $\beta$ is one as well. By Lemma \ref{lem:uniqueness} it follows that $\tilde{f}_i$ is an isomorphism, completing the induction argument.
\end{proof}

\section{Homotopy Color Lie Algebra}

We start this section by giving the definition of graded color Lie algebra over an associative ring $R$, 
regardless of the characteristic of $R$. Since the Lie algebras of interest to us arise
in cohomology, we adopt cohomological conventions for our gradings.  We continue to use
the same notational conventions appearing in Remark \ref{rem:threeGradings2}. 

In our applications the Lie algebras will come equipped with an internal grading which, for the sake of readability, will be dropped in the definition of graded color Lie algebra.

Let $G$ be an abelian group, with identity $e_G$, and let $\e$ be a skew bicharacter on $G$ as defined
in Section \ref{sec:AcyClos}. Let $L=\bigoplus_{(i,\sigma)\in \mathbb{Z}\times G}L^{i,\sigma}$ be a 
$(\mathbb{Z}\times G)$-graded left $R$-module.  If $R$ comes equipped with an internal grading then $L$
is assumed to also have an internal grading, in which case all the maps in this section are assumed
to be compatible with respect to this grading and all the elements that follow are assumed to be
homogeneous. As before, if $x\in L^{i,\sigma}$
and $y\in L^{j,\tau}$ then we abuse notation and write $\e(x,y)$ for 
$\e(\sigma,\tau)$. If $x\in L^{i,\sigma}$ then we denote by 
$|x|$ its cohomological degree, i.e. $|x|=i$.

\begin{definition} \label{def:CLA}
The $(\mathbb{Z}\times G)$-graded $R$-module $L$ is said to be a
\emph{graded color Lie algebra} if it is endowed with a $R$-bilinear operation
\[
[-,-]:L\times L\rightarrow L,
\]
and \emph{square} maps
\[
(-)^{[2]}:L^{2i+1,\sigma}\rightarrow L^{4i+2,\sigma^2}, 
\]
such that for all $x\in L^{i,\sigma}$ and $y\in L^{j,\tau}$, one has
\begin{enumerate}
\item $[L^{i,\sigma},L^{j,\tau}]\subseteq L^{i+j,\sigma\tau}$,
\item The bracket is color anti-commutative:
$$[x,y]=-(-1)^{|x||y|}\e(x,y)[y,x],$$
\item The color Jacobi identity holds:
$$(-1)^{|z||x|}\e(z,x)[[x,y],z]+
 (-1)^{|x||y|}\e(x,y)[[y,z],x]+
 (-1)^{|y||z|}\e(y,z)[[z,x],y]
 =0,$$
\item if $|x|$ is even then $[x,x]=0$,
\item if $|x|$ is odd then $[[x,x],x]=0$,
\item if $x,y\in L^{i,\sigma}$ with $i$ odd then $(x+y)^{[2]}=[x,y]+x^{[2]}+y^{[2]}$,
\item if $|x|$ is odd and $a\in R$ then $(ax)^{[2]}=a^2x^{[2]}$,
\item if $|x|$ is odd then $[x^{[2]},y]=[x,[x,y]]$ for all $y\in L$.
\end{enumerate}
\end{definition}

\begin{remark}
We point out that if the characteristic of $R$ is not 2 nor 3 it is possible to give the previous 
definition using a skew bicharacter defined on the group $\mathbb{Z}\times G$. 
\end{remark}
\begin{remark} \label{AdIsColorDer}
Property (3) in the previous definition can also be expressed as
\[
[x,[y,z]]=[[x,y],z]+(-1)^{|x||y|}\e(x,y)[y,[x,z]].
\]
This is equivalent to saying the map
$\operatorname{ad}_x(y) = [x,y]$ is a color derivation of $L$.
\end{remark}

\begin{remark}
Let $B$ be an associative $\bbZ\times G$-graded $R$-algebra. Then $\Lie(B)$ is the graded color Lie algebra with underlying $R$-module $B$, bracket given by $[x,y]=xy-(-1)^{|x||y|}\e(x,y)yx$, with $x$ and $y$ trihomogeneous and $z^{[2]}=z^2$ for $z$ trihomogeneous with $|z|$ odd. One can check that $\Lie(B)$ is indeed a graded color Lie algebra.
\end{remark}

\begin{definition}
Let $L,L'$ be graded color Lie algebras over $R$. A trihomogeneous map of left $R$-modules $f:L\rightarrow L'$ is a morphism of graded color Lie algebras if $f([x,y])=[f(x),f(y)]$ for all $x,y\in L$ and $f(z^{[2]})=f(z)^{[2]}$ for all trihomogeneous $z\in L$ with $|z|$ odd.
\end{definition}

\begin{definition}
Let $L$ be a graded color Lie algebra. The \emph{universal enveloping algebra of $L$} is
the following quotient of the tensor algebra $T(L)$:
\[
U(L)=\frac{T(L)}{\begin{pmatrix}
x\otimes y-(-1)^{|x||y|}\e(x,y)y\otimes x-[x,y]\\ z\otimes z-z^{[2]},|z|\mathrm{\;odd}
\end{pmatrix}}.
\]
\end{definition}

\begin{remark}
The universal enveloping algebra of $L$ satisfies the following universal property.
Given a $\mathbb{Z}\times G$-graded associative $R$-algebra $B$ and morphism of 
graded color Lie algebras $g:L\rightarrow\Lie(B)$, there is a unique 
homomorphism of $\mathbb{Z}\times G$-graded associative algebras
$g':U(L)\rightarrow B$, such that $g=g'f$, where $f$ is the canonical inclusion
$f : L \to U(L)$.  We call $g'$ the \emph{universal extension} of $g$.
\end{remark}

\begin{remark}
Assume that $L^n=0$ for $n\leq0$, and that $L$ is a free $R$-module.
Fix a trihomogeneous basis of $L$, denoted by $\Theta=\{\theta_i\}_{i\geq 1}$ ordered
in such a way that $|\theta_i|\leq|\theta_j|$ for $i<j$. Let 
$I=(i_1,i_2,\ldots)$ be a sequence of nonnegative integers such that $i_j\in\{0,1\}$ if $|\theta_j|$ 
is odd and $i_j=0$ for $j\gg0$. Fix an indexing sequence $I$ and a $q$ such that $i_j=0$ for $j>q$, a 
\emph{normal monomial on $\Theta$} is an element of $U(L)$ of the form 
$\theta^I=\theta_q^{i_q}\cdots\theta_1^{i_1}$ (we are dropping the tensor product sign). It is a 
straightforward check that the set of normal monomials on $\Theta$ span $U(L)$. 

\end{remark}

\begin{definition}
A \emph{color DG Lie algebra} over a ring $R$ is a graded color Lie algebra $L$ over $R$ with a degree $(-1,e_G)$ $R$-linear map $\partial: L\rightarrow L$, such that $\partial^2=0$ and 
\begin{equation}\label{DGCLA}
\partial([x,y])=[\partial(x),y]+(-1)^{|x|}[x,\partial(y)],\;\mathrm{and}\;\partial(z^{[2]})=[\partial(x),z]\;\mathrm{for}\;|z|\;\mathrm{odd}.
\end{equation}
\end{definition}
Our interest in color DG Lie algebras arises from the following lemma.
For the remainder of the paper $R$ will denote the quotient of a skew polynomial ring by an ideal
generated by a sequence of homogeneous normal elements.  
\begin{lemma}
Let $R\rightarrow \Ay{R}{Y}$ be a semi-free extension. The inclusion
\[
\Der_R(\Ay{R}{Y},\Ay{R}{Y})\subseteq \Lie(\Hom_R(\Ay{R}{Y},\Ay{R}{Y}))
\]
is one of color DG Lie algebras.
\end{lemma}
\begin{proof}
\begingroup
\allowdisplaybreaks
Let $x,y$ be trihomogeneous elements of $\Ay{R}{Y}$ and $\theta$ be a derivation of odd degree. Then
\begin{align*}
\theta^2(xy)&=\theta(\theta(x)y+(-1)^{|x|}\e(\theta,x)x\theta(y))\\
&=\theta^2(x)y+(-1)^{|\theta|(|\theta|+|x|)}\e(\theta,\theta(x))\theta(x)\theta(y)+\\
&\hphantom{=}(-1)^{|x|}\e(\theta,x)(\theta(x)\theta(y)+(-1)^{|\theta||x|}\e(\theta,x)x\theta^2(y))\\
&=\theta^2(x)y-(-1)^{|x|}\e(\theta,\theta)\e(\theta,x)\theta(x)\theta(y)+\\
&\hphantom{=}(-1)^{|x|}\e(\theta,x)\theta(x)\theta(y)+\e(\theta,x)\e(\theta,x)x\theta^2(y)\\
&=\theta^2(x)y+\e(\theta^2,x)x\theta^2(y).
\end{align*}
If $|y|$ is an even variable then
\[
\theta^2(y^{(i)})=\theta(\theta(y)y^{(i-1)})=\theta^2(y)y^{(i-1)}-(\theta(y))^2y^{(i-2)}=\theta^2(y)y^{(i-1)}.
\]
Now we prove that the bracket of two color derivations is a color derivation
\begin{align*}
[\theta,\xi](xy)&=(\theta\xi-(-1)^{|\theta||\xi|}\e(\theta,\xi)\xi\theta)(xy)\\
&=\theta(\xi(x)y+(-1)^{|\xi||x|}\e(\xi,x)x\xi(y))\\
&\hphantom{=}-(-1)^{|\theta||\xi|}\e(\theta,\xi)\xi(\theta(x)y+(-1)^{|\theta||x|}\e(\theta,x)x\theta(y))\\
&=\theta\xi(x)y+(-1)^{|\theta||\xi(x)|}\e(\theta,\xi(x))\xi(x)\theta(y)\\
&\hphantom{=}+(-1)^{|\xi||x|}\e(\xi,x)(\theta(x)\xi(y)+(-1)^{|\theta||x|}\e(\theta,x)x\theta\xi(y))\\
&\hphantom{=}-(-1)^{|\theta||\xi|}\e(\theta,\xi)(\xi\theta(x)y+(-1)^{|\xi||\theta(x)|}\e(\xi,\theta(x))\theta(x)\xi(y)\\
&\hphantom{=}+(-1)^{|\theta||x|}\e(\theta,x)(\xi(x)\theta(y)+(-1)^{|\xi||x|}\e(\xi,x)x\xi\theta(y)))\\
&=\theta\xi(x)y+(-1)^{|x|(|\theta|+|\xi)|}\e(\theta\xi,x)x\theta\xi(y)\\
&\hphantom{=}-(-1)^{|\theta||\xi|}\e(\theta,\xi)(\xi\theta(x)y+(-1)^{|x|(|\xi|+|\theta|)}\e(\theta,x)\e(\xi,x)x\xi\theta(y))\\
&=\theta\xi(x)y-(-1)^{|\theta||\xi|}\e(\theta,\xi)\xi\theta(x)y\\
&\phantom{=}+(-1)^{|x|(|\theta|+|\xi|)}\e(\theta\xi,x)x(\theta\xi(y)-(-1)^{|\theta||\xi|}\e(\theta,\xi)\xi\theta(y))\\
&=[\theta,\xi](x)y+(-1)^{|x|(|\theta|+|\xi|)}\e(\theta\xi,x)[\theta,\xi](y) 
\end{align*}
This shows that the derivations form a graded color Lie algebra, we now prove that they form a DG color Lie algebra. We denote by $\partial^{\Ay{R}{Y}}$ the differential of the acyclic closure and by $\partial^{\Der}$ the differential of the complex of derivations. We notice that $\partial^{\Ay{R}{Y}}$ is a color derivation and that if $\theta$ is a color derivation then $\partial^{\Der}(\theta)=[\partial^{\Ay{R}{Y}},\theta]$. Now the conditions
in \eqref{DGCLA} follow from the color Jacobi identity (using Remark \ref{AdIsColorDer}) and property (8) in the definition of graded color Lie algebras.
\endgroup
\end{proof}

Since the variables adjoined in the acyclic closure of $\kk$ over $R$ are trigraded, and since
the acyclic closure is unique up to isomorphism, we obtain a family of invariants of a color commutative
algebra $R$.

\begin{definition}
Let $\Ay{R}{Y}$ be an acyclic closure of $\kk$ over $R$. The following invariants of $R$ are called the \emph{deviations} of $R$:
\begin{align*}
\varepsilon_{i,\sigma,j}(R)=&|\{y\in Y\mid |y|=i,\gdeg{y}=\sigma,\mathrm{deg}\;y=j,\}|,&~&i,j\in\mathbb{N},\sigma\in G.
\end{align*}
\end{definition}

\begin{definition}\label{defn:piR}
The \emph{homotopy color Lie algebra} of $R$ is 
\[
\pi(R)=\mathrm{H}(\Der_R(\Ay{R}{Y},\Ay{R}{Y})),
\]
where $\Ay{R}{Y}$ is an acyclic closure of $\kk$ over $R$.  Note that since the acyclic
closure is trigraded, $\pi(R)$ possesses an additional internal grading which is not part
of our definition of graded color Lie algebra above.  Also, it is an invariant of $R$
by Theorem \ref{thm:uniqueness}.  We denote its graded components as
\[
\pi^{i,\sigma,j}(R),
\]
where $i$ denotes the cohomological degree, $\sigma$ denotes the group degree and
$j$ the internal degree.
\end{definition}

\begin{theorem}\label{thm:BasisOfPi}
Let $\Ay{R}{Y}$ be an acyclic closure of $\kk$ over $R$, where $Y=\{y_i\}_{i\geq 1}$ and $|y_i|\leq |y_j|$ for $i<j$.
\begin{enumerate}
\item $\mathrm{rank}_{\kk}\;\pi^{i,\sigma,j}(R)=\varepsilon_{i,\sigma,j}(R)$, for $i,j\in\mathbb{N}$ and $\sigma\in G$.
\item $\pi(R)$ has a $\kk$-basis
\[
\Theta=\{\theta_i=\cls{\vartheta_i}\mid\vartheta_i\in\Der_R(\Ay{R}{Y},\Ay{R}{Y}),\vartheta_i(y_j)=\delta_{ij}\;\mathrm{for}\;j\leq i\}_{i\geq 1}.
\]
\end{enumerate}

\end{theorem}

\begin{proof}
For the first claim, notice that
$(\Diff_R\Ay{R}{Y})\otimes_{\Ay{R}{Y}}\kk\cong \kk Y$ as complexes with trivial 
differential by Proposition \ref{prop:Diff}. There is a chain of graded 
isomorphisms
\begin{align*}
\pi(R)&=\mathrm{H}(\Der_R(\Ay{R}{Y},\Ay{R}{Y}))\\
&\cong\mathrm{H}(\Der_R(\Ay{R}{Y},\kk))\\
&\cong\mathrm{H}(\Hom_{\Ay{R}{Y}}(\Diff_R\Ay{R}{Y},\kk))\\
&\cong\mathrm{H}(\Hom_{\kk}((\Diff_R\Ay{R}{Y})\otimes_{\Ay{R}{Y}}\kk,\kk))\\
&\cong\mathrm{H}(\Hom_{\kk}(\kk Y,\kk))\\
&=\Hom_{\kk}(\kk Y,\kk)
\end{align*}
where the first isomorphism follows by Corollary \ref{cor:derQuism}, the second by 
Proposition \ref{prop:Diff}, the third by Proposition \ref{prop:adjoint}, and the 
fourth by the observation at the start of the proof. Now we are done by definition of deviation.

For the second claim, by Lemma \ref{lem:UpTriang}(1) there are derivations $\vartheta_i$ such that $\vartheta_i(y_j)=\delta_{ij}$ for $j\leq i$. If $\vartheta_{i_1},\ldots,\vartheta_{i_m}$ all have the same (tri-)degree and $\sum_{k}\beta_k\theta_{i_k}=0$ for some $\beta_k\in\kk$, then by evaluating at $y_{i_j}$ we deduce that $\beta_j=0$ for all $j=1,\ldots,m$. This proves that $\Theta$ is a linearly independent set. By part (1) we know that in each (tri-)degree $\Theta$ has the same dimension of $\pi(R)$ and therefore forms a $\kk$-basis.
\end{proof}

\section{Ext Algebra} \label{sec:extAlgebra}

In this section, we study the graded $\kk$-algebra $\Ext_R(\kk,\kk)$, which is the homology
of the DG algebra $\Hom_R(\Ay{R}{Y},\Ay{R}{Y})$.  In what follows, we denote the augmentation map 
$\Ay{R}{Y}\rightarrow\kk$ by $\epsilon$.

\begin{remark} \label{rem:leftExtIsRightExtOp}
Since $\Hom_R(\Ay{R}{Y},\Ay{R}{Y})$ denotes the set of left color $R$-linear maps from $\Ay{R}{Y}$
to itself, Remark \ref{rem:rightLinear} shows that the Ext algebra mentioned above is $\Ext_R(\kk_R,\kk_R)$,
the Ext algebra of $\kk$ over $R$ where $\kk$ is considered a \emph{right} $R$-module.

However, there are isomorphisms (cf. \cite[Pg. 5]{PnP}) 
\[
\Ext_R({}_R\kk,{}_R\kk) \cong \Ext_{R^{\op}}({}_{R^{\op}}\kk,{}_{R^{\op}}\kk)^{\op} \cong \Ext_R(\kk_R,\kk_R)^{\op}
\]
as graded $\kk$-algebras so that one may convert the descriptions of right Ext algebras given below
into a description of left Ext algebras by taking the opposite ring.  This is especially important when
comparing our results with those that exist in the literature.
\end{remark}

\begin{theorem}
The inclusion $\Der_R(\Ay{R}{Y},\Ay{R}{Y})\rightarrow\Hom_R(\Ay{R}{Y},\Ay{R}{Y})$ induces an injective morphism of graded color Lie algebras
\[
\iota:\pi(R)\rightarrow\Lie(\Ext_R(\kk,\kk)).
\]
\end{theorem}

\begin{proof}
We consider the following diagram
\begin{equation}\label{DerHomDiag}
\begin{tikzpicture}[baseline=(current  bounding  box.center)]
 \matrix (m) [matrix of math nodes,row sep=3em,column sep=4em,minimum width=2em] {
 \Der_R(\Ay{R}{Y},\Ay{R}{Y})&\Hom_R(\Ay{R}{Y},\Ay{R}{Y})\\
 \Der_R(\Ay{R}{Y},\kk)&\Hom_R(\Ay{R}{Y},\kk)\\};
 \path[->] (m-1-1) edge (m-1-2);
 \path[->] (m-2-1) edge (m-2-2);
 \path[->] (m-1-1) edge node[left] {$\Der_R(\Ay{R}{Y},\epsilon)$} node[right] {$\simeq$} (m-2-1);
 \path[->] (m-1-2) edge node[right] {$\Hom_R(\Ay{R}{Y},\epsilon)$} node[left] {$\simeq$} (m-2-2);
\end{tikzpicture}
\end{equation}
where the top and bottom maps are just inclusions and the left and right maps are the quasi-isomorphisms given by Corollary \ref{cor:derQuism} and Proposition \ref{prop:homPreservesQuism} respectively. A straightforward computation shows that this diagram is commutative. By taking (co-)homology in the diagram it follows that the bottom map is (isomorphic to) $\iota$, which is therefore injective.
\end{proof}

We prove a version of the Poincar\'{e}-Birkhoff-Witt Theorem for the color Lie 
algebra $\pi(R)$. In the next theorem we will use the same notation used in
Theorem \ref{thm:BasisOfPi}.

\begin{theorem}\label{thm:Ext=Upi}
The normal monomials on $\Theta$ form a $\kk$-basis of $U\pi(R)$. Moreover the universal extension of the map
\[
\iota:\pi(R)\rightarrow\Lie(\Ext_R(\kk,\kk)),
\]
is an isomorphism of associative algebras
\[
\iota':U\pi(R)\rightarrow\Ext_R(\kk,\kk).
\]

\end{theorem}

\begin{proof}
We identify $\iota$ with the inclusion
\[
\iota:\Der_R(\Ay{R}{Y},\kk)\rightarrow\Lie(\Hom_R(\Ay{R}{Y},\kk)).
\]
By Theorem \ref{thm:BasisOfPi}, a $\kk$-basis of $\Der_R(\Ay{R}{Y},\kk)$ is given by 
\[
\Theta=\{\epsilon\vartheta_i\mid\vartheta_i\in\Der_R(\Ay{R}{Y},\Ay{R}{Y}),\vartheta_i(y_j)=\delta_{ij}\;\mathrm{for}\;j\leq i\}_{i\geq 1}.
\]
Since $\Hom_R(\Ay{R}{Y},\kk)$ is the graded $\kk$-dual of $\Ay{R}{Y}$, the ``dual elements'' to the normal monomials of the acyclic  closure form a $\kk$-basis.

We will use $H$ and $I$ to denote indexing sequences of normal monomials (in both $U\pi(R)$
and $\Ay{R}{Y}$). Denoting $\epsilon\vartheta_i$ by
$\theta_i$, let $\theta^I\in U\pi(R)$ be a normal monomial.  
Note that since $\iota'$ is the universal extension of the inclusion $\iota$,
$\iota'$ sends a normal monomial $\theta^I$ to itself. By Lemma \ref{lem:UpTriang},
$\theta^I(y^H)=0$ if $H<I$ and $\theta^I(y^H) = 1$ if $H=I$. 
Therefore the coordinate vectors of 
the normal monomials on the elements of $\Theta$ with respect to the dual basis of 
the normal monomials of the acyclic closure are linearly independent. We had 
previously noted that normal monomials span, hence they are a $\kk$-basis of
$U\pi(R)$.

To prove that $\iota'$ is an isomorphism we first notice that it is injective since
the images of the normal monomials on $\Theta$ are $\kk$-linearly independent. By 
Theorem \ref{thm:BasisOfPi}(1) and since the normal monomials on $\Theta$ are
a $\kk$-basis of $U\pi(R)$, we deduce that in each degree the algebras $U\pi(R)$ 
and $\Ext_R(\kk,\kk)$ have the same $\kk$-dimension, so that $\iota'$ is an 
isomorphism.
\end{proof}

We recall the definition of graded color Hopf algebra. In our applications the Hopf algebras will come equipped with an internal grading which, for the sake of readability, will be dropped in the definition of color Hopf algebra. 
\begin{definition}
Let $G$ be an abelian group with a bicharacter $\e:G\times G\rightarrow\kk^*$. Let $H=\bigoplus_{(i,\sigma)\in \mathbb{Z}\times G} H^{(i,\sigma)}$ be a 
$(\mathbb{Z}\times G)$-graded connected $\kk$-algebra with product $m$ and unit $u$.  We denote
$m(a\otimes b)$ by $ab$.  If $H$ comes equipped with an internal grading, then all the maps that follow in this definition are assumed to be compatible with respect to this grading and all the elements that follow are assumed to be homogeneous. If $a\in H^{(i,\sigma)}$, then we denote by $|a|$ its 
$\mathbb{Z}$-degree, i.e. $|a|=i$.  If $H$ is a $(\mathbb{Z}\times G)$-graded coalgebra with coproduct $\Delta$ and counit 
$\epsilon$, then we say that $H$ is a \emph{graded color coalgebra} if
\[
\epsilon(b)=\e(a,b)\epsilon(b),\mathrm{and}\;
\epsilon(a)=\e(a,b)\epsilon(a),\quad\mathrm{for\;all\;}a\in H^{(i,\sigma)}, b\in H^{(j,\tau)}.
\]
We let $H\otimes^\e H$ be the $(\bbZ\times G)$-graded algebra which is
$H \otimes H$ as a vector space, with product given by the following,
for $a\in H^{(i,\sigma)},b\in H^{(j,\tau)}, c\in H^{(k,\rho)}, d\in H^{(l,\gamma)}$:
\[
(a\otimes b)(c\otimes d)=(-1)^{|b||c|}\e(b,c)(ac)\otimes (bd).
\]

The algebra $H$ is a \emph{graded color bialgebra} if $\Delta : H \to H \otimes^\e H$ and $\epsilon$
are maps of $(\mathbb{Z}\times G)$-graded algebras.

A \emph{graded color Hopf algebra} is a graded color bialgebra with an antipode map $S$, i.e. with a map $S:H\rightarrow H$ such that
\[
m(S\otimes \id_H)\Delta=u\epsilon=m(\id_H\otimes S)\Delta.
\]

\end{definition}

\begin{remark}
A color Hopf algebra is just a special case of the more general notion
of braided Hopf algebra; see \cite{Khar}.
\end{remark}

\begin{remark}
Let $\pi(R)$ be the graded color Lie algebra of Definition \ref{defn:piR}. Then $U\pi(R)$ is a graded color Hopf algebra with the following structure:
\begin{align*}
\Delta(x)=x\otimes1+1\otimes x,& \quad x\in\pi(R),\\
\epsilon(x)=0,& \quad x\in\pi(R),\\
S(x)=-x,& \quad x\in\pi(R).
\end{align*}
Where $\Delta$ and $\epsilon$ are extended to all of $U\pi(R)$ multiplicatively and $S$ is extended to all of $U\pi(R)$ color anti-multiplicatively, i.e.
\[
S(ab)=(-1)^{|a||b|}\e(a,b)S(b)S(a),\quad a,b\in U\pi(R)\;\mathrm{trihomogeneous}.
\]
\end{remark}

A remarkable consequence of Theorem \ref{thm:Ext=Upi} and the previous remark is
\begin{corollary}
The algebra $\Ext_R(\kk,\kk)$ is a graded color Hopf algebra.
\end{corollary}

\section{Lie operations on $\pi^1(R)$}\label{sec:pi}
In this section, we carry out the computations necessary to compute the
bracket on $\pi(R)$ in homological degree one and obtain results analogous
to those of Sj\"odin \cite{Sjo}.  The theorem statement will come at the
end of the section, after all the necessary notation has been introduced.

Recall that $R=Q/I$ with $I=(f_1,\ldots,f_c)$ a homogeneous ideal generated by normal elements.  We also assume that $I \subseteq R_{\geq 2}$,
and therefore for each $j = 1,\dots,c$, there exist homogeneous and $G$-homogeneous elements $a_{h,i,j} \in Q$ such that
\[
f_j=\sum_{1\leq h\leq i\leq n}a_{h,i,j}x_hx_i.
\]
If $f\in Q$ then we denote the image in $R$ by $\bar{f}$.
Let $K^R(\overline{\mathbf{x}})$ be the Koszul complex on $\bar{x}_1,\ldots,\bar{x}_n$
as in Definition \ref{def:koszulComplex}, i.e.,
\[
K^R(\overline{\mathbf{x}})=\Ay{R}{y_1,\ldots,y_n\mid\partial (y_i)=\bar{x}_i}.
\]
Let $T_2$ be the complex that one obtains from $K^R(\bar{\mathbf{x}})$
by killing a minimal generating set of $\mathrm{H}_1(K^R(\overline{\mathbf{x}}))$, 
i.e.,
\[
T_2=\Ay{K^R(\overline{\mathbf{x}})}{y_{n+1},\ldots,y_{n+c}\mid\partial (y_{n+j})=\sum_{1\leq h\leq i\leq n}\bar{a}_{h,i,j}\bar{x}_hy_i}.
\]
Adjoining the variables $y_1,\dots,y_{n+c}$ to $R$ are the first two steps
in constructing the acyclic closure of $\kk$ over $R$, which we
denote by $\Ay{R}{Y}$.  We continue our convention of numbering the variables 
of an acyclic closure in a manner which respects the homological
degree.

Let $\vartheta_l\in\Der_R(K^R(\overline{\mathbf{x}}),K^R(\overline{\mathbf{x}}))$
be such that $\vartheta_l(y_j)=\delta_{l,j}$ for $j=1,\ldots,n$.
To extend $\vartheta_l$ to a derivation of $T_2$, 
notice that:
\begin{equation*}
\vartheta_l\partial(y_{n+j})=\vartheta_l(\sum_{h\leq i}\bar{a}_{h,i,j}\bar{x}_hy_i)
=\sum_{h=1}^l\bar{a}_{h,i,j}\bar{x}_h
=\partial(\sum_{h=1}^l\bar{a}_{h,i,j}y_h).
\end{equation*}
Therefore by setting $\vartheta_l(y_{n+j})=-\sum_{h=1}^l\bar{a}_{h,i,j}y_h$ for $l=1,\ldots,n$ and 
$j=1,\ldots,c$ (and extending so that the color Leibniz rule holds) we obtain an extension to a 
derivation of $T_2$ that commutes with the differential. We notice that 
$\gdeg{\vartheta_l}=\gdeg{y_l}^{-1}$ and therefore, if $\gdeg{y_l}=\sigma_l$ and
$\gdeg{y_i}=\sigma_i$ we have that $\e(\sigma_l,\sigma_i)=\e(\sigma_l^{-1},\sigma_i^{-1})$, i.e.\
$\e(\vartheta_l,\vartheta_i)=\e(y_l,y_i)$. We compute $[\vartheta_l,\vartheta_i](y_{n+j})$ for $l<i$ and
$j=1,\ldots,c$:
\begin{align*}
[\vartheta_l,\vartheta_i](y_{n+c})&=(\vartheta_l\vartheta_i+\e(\vartheta_l,\vartheta_i)\vartheta_i\vartheta_l)(y_{n+c})\\
&=\vartheta_l(-\sum_{h=0}^i\bar{a}_{h,i,j}y_h)+\e(\vartheta_l,\vartheta_i)\vartheta_i(-\sum_{h=0}^l\bar{a}_{h,l,j}y_h)=-\bar{a}_{l,i,j}.
\end{align*}
For the square we have
\begin{align*}
\vartheta_i\vartheta_i(y_{n+c})&=\vartheta_i(-\sum_{h=0}^i\bar{a}_{h,i,j}y_h)=-\bar{a}_{i,i,j}.
\end{align*}
We collect the previous results in the following theorem:
\begin{theorem}
Let $Q = \kk_\q[x_1,\dots,x_n]$ be a skew polynomial ring,
$R = Q/I$ with $I = (f_1,\dots,f_c) \subseteq Q$ an ideal with each
$f_i$ normal, homogeneous and of internal degree at least two, and let $\epsilon$ denote the augmentation from $R$ to $\kk$.
For each $j$, write $f_j = \sum_{1\leq h\leq i\leq n} a_{h,i,j} x_hx_i$ for normal, homogeneous $a_{h,i,j} \in Q$.
Let $\Ay{R}{Y}$ be the acyclic closure of $\kk$ over $R$, and let
$\theta_i = \epsilon\vartheta_i \in \Der_R(\Ay{R}{Y},\kk)$, where $\vartheta_i$ is the derivation corresponding
to the variable $y_i$.  Then for all $1 \leq l < i \leq n$, one has equalities:
\begin{equation}\label{eqn:Pi1}
[\theta_l,\theta_i]=-\sum_{j=1}^c\epsilon(\bar{a}_{l,i,j})\theta_{n+j},\quad\mathrm{for}\;l<i\quad\mathrm{and}\quad \theta_i^{[2]}=-\sum_{j=1}^c\epsilon(\bar{a}_{i,i,j})\theta_{n+j}.
\end{equation}
\end{theorem}

\section{Skew Complete Intersections} \label{sec:skewCIs}

\begin{definition}\label{defn:QCI}
We say that the ring $Q/I$ is a \emph{skew complete intersection} if $Q$ is a skew polynomial ring and $I$ is a two-sided ideal generated by a regular sequence of homogeneous normal elements.
\end{definition}
\begin{remark}
The definition of quantum complete intersection appearing in \cite{BerOpp} requires the ideal $I$ to be generated by powers of the variables of $Q$.
Definition \ref{defn:QCI} generalizes this definition.
\end{remark}
\begin{definition}
Let $\q$ be a multiplicatively antisymmetric matrix. A \emph{skew exterior algebra} is an algebra of the form
\[
\qwedge\kk^n=\frac{\kk_\q[x_1,\ldots,x_n]}{(x_1^2,\ldots,x_n^2)}.
\]
We consider it a DG algebra with zero differential and graded cohomologically
with $|x_i|=1$ for all $i$'s.
\end{definition}

\begin{theorem}\label{thm:ResQCI}
Let $R=Q/I$ be a skew complete intersection with $I$ generated by
$\{f_1,\ldots,f_c\}$ with each $f_i$ normal, homogeneous of internal degree at least two.
Let
\[
K^R(\overline{\mathbf{x}})=\Ay{R}{y_1,\ldots,y_n\mid\partial (y_i)=\bar{x}_i}.
\]
If $f_j=\sum_{i}a_{i,j}x_i$, then $\mathrm{H_1}(K^R(\overline{\mathbf{x}}))$ is generated by the cycles $\sum_{i=1}^n\bar{a}_{i,j}y_i$ for $j=1,\ldots,c$. Moreover an acyclic closure of $\kk$ over $R$ is
\begingroup
\arraycolsep=1.4pt
\[
\Ay{R}{Y}=\Ay{R}{y_1,\ldots,y_{n+c}\left| \begin{array}{llll}\partial  (y_i) & = & \bar{x}_i & ~\mathrm{for}\;i=1,\ldots,n, \\
\partial (y_{n+j}) & = & \sum_{i=1}^n\bar{a}_{i,j}y_i & ~\mathrm{for}\;j=1,\ldots,c\end{array}\right.}.
\]
\endgroup
\end{theorem}

\begin{proof}
We denote the sequence $f_1,\ldots,f_c$ by $\mathbf{f}$ and the sequence $x_1,\ldots,x_n$ by $\mathbf{x}$. Let $\alpha$ be the
quasi-isomorphism $\alpha: K^Q(\mathbf{f})\rightarrow R$ and $\beta$ the quasi-isomorphism $\beta:K^Q(\mathbf{x})\rightarrow\kk$.
By Proposition \ref{prop:homPreservesQuism}, these maps induce quasi-isomorphisms
\[
\kk\otimes_Q K^Q(\mathbf{f})\xleftarrow{\beta\otimes K^Q(\mathbf{f})} K^Q(\mathbf{x})\otimes_Q K^Q(\mathbf{f})\xrightarrow {K^Q(\mathbf{x})\otimes\alpha}K^Q(\mathbf{x})\otimes_QR.
\]
We notice that $K^Q(\mathbf{x})\otimes_QR \cong K^R(\overline{\mathbf{x}})$, the skew Koszul complex of $\bar{x}_1,\ldots,\bar{x}_n$ over $R$, while
$\kk\otimes_Q K^Q(\mathbf{f})$ is a skew exterior algebra, which we denote by $\Lambda$. We fix the following notation:
\[
K^Q(\mathbf{f})=\Ay{Q}{e_1,\ldots,e_c\mid\partial (e_i)= f_i},\quad\mathrm{and}\quad K^Q(\mathbf{x})=\Ay{Q}{y_1,\ldots,y_n\mid \partial (y_i)=x_i}.
\]
The element $\sum_{i=1}^na_{i,j}y_i\otimes1-1\otimes e_j$ is mapped by $K^Q(\mathbf{x})\otimes\alpha$ to 
$\sum_{i,j}a_{i,j}y_i\otimes1$ which corresponds to the element 
$\sum_{i=1}^n\bar{a}_{i,j}y_i$ of the Koszul complex of $\bar{x}_1,\ldots,\bar{x}_n$ over $R$.
That same element is mapped by $\beta\otimes K^Q(\mathbf{f})$ to $-1\otimes e_j$ which 
corresponds to the element $-e_j$ thought as one of the variables generating $\Lambda$. This shows
that the homology of the skew Koszul complex of $\bar{x}_1,\ldots,\bar{x}_n$ is 
generated in cohomological degree 1 by the cycles $\sum_{i=1}^n\bar{a}_{i,j}y_i$ for $j=1,\ldots,c$,
proving the first part of the theorem. These cycles are also regular because they
correspond to the variables (with a negative sign) of $\Lambda$. By Theorem \ref{thm:RegEl} once these
cycles are killed we obtain a resolution of $\kk$, proving the last assertion
of the theorem.
\end{proof}

As a consequence of the proof, we obtain the following:
\begin{corollary}
If $R=Q/I$ is a skew complete intersection then its Koszul homology algebra is isomorphic to a skew exterior algebra.
\end{corollary}

\begin{remark}\label{rmk:BasisOfPiQCI}
Let $\theta_1,\ldots,\theta_n$ be a $\kk$-basis of $\pi^{1}(R)$ and $\theta_{n+1},\ldots,\theta_{n+c}$ be a $\kk$-basis for
$\pi^{2}(R)$. Then by Theorem \ref{thm:BasisOfPi}(1) and Theorem \ref{thm:ResQCI} these elements form a $\kk$-basis for the color Lie 
algebra $\pi(R)$. 
\end{remark}

With the notation from Section \ref{sec:pi}, we also have the following description of the Ext algebra of a skew complete intersection ring.

\begin{theorem}\label{thm:ExtPres}
If $R$ is a skew complete intersection then, as a graded color Hopf algebra
\[
\Ext_R(\kk,\kk)\cong \frac{T(\kk\theta_1 \oplus \cdots \oplus \kk\theta_{n+c})}{
\left(\begin{array}{ll}[\theta_l,\theta_i]+\sum_{j=1}^c\epsilon(\bar{a}_{l,i,j})\theta_{n+j}, & \mathrm{for}\;l<i\leq n\\
\theta_i^2+\sum_{j=1}^c\epsilon(\bar{a}_{i,i,j})\theta_{n+j}, & \mathrm{for}\;i\leq n \\
\left[\theta_l,\theta_i\right], & \mathrm{for}~l~\mathrm{or}~i > n \\
\end{array}\right)},
\]
where $[\theta_l,\theta_i]=\theta_l\theta_i-(-1)^{|\theta_l||\theta_i|}\e(\theta_l,\theta_i)\theta_i\theta_l$, with $|\theta_i|=1$ if $i\leq n$ and 2 otherwise,
and the Hopf structure on the right is obtained by identifying it with $U\pi(R)$.
\end{theorem}
\begin{proof}
It follows from Theorem \ref{thm:Ext=Upi}, formula \eqref{eqn:Pi1} and Remark \ref{rmk:BasisOfPiQCI}.
\end{proof}
Using Remark \ref{rem:leftExtIsRightExtOp} one sees that Theorem \ref{thm:ExtPres} generalizes \cite[Theorem 5.3]{BerOpp}.

If $P^R_\kk(t)$ denotes the (ungraded) Poincar\'{e} series of $\kk$ over $R$, i.e.\ the Hilbert series of $\Ext_R(\kk,\kk)$, then as a corollary of Theorem \ref{thm:ResQCI} and \cite[Theorem 7.1.3]{IFR} we deduce
\begin{corollary}\label{cor:PoincSeries}
If $R=Q/I$ is a skew complete intersection, with $Q$ skew polynomial ring in $n$ variables $x_1,\ldots,x_n$ and $I=(f_1,\ldots,f_c)\subseteq(x_1,\ldots,x_n)^2$, then
\[
P^R_\kk(t)=\frac{(1+t)^n}{(1-t^2)^c}.
\]
\end{corollary}
The invariant defined below captures the growth of the minimal free resolution of $\kk$ over a connected graded $\kk$-algebra $\Omega$ and it is closely related to the Gelfand-Kirillov dimension of $\Ext_\Omega(\kk,\kk)$.
\begin{definition}
Let $\Omega$ be a connected graded $\kk$-algebra. The \emph{complexity} of $\kk$ over $\Omega$ is
\[
\mathrm{cx}_\Omega\;\kk=\mathrm{inf}\{d\in\mathbb{N}\mid \exists f(t)\in\mathbb{Z}[t],\;\mathrm{deg}\;f=d-1,\;\mathrm{dim}_\kk\Ext_\Omega^n(\kk,\kk)\leq f(n)\mathrm{\;for\;}n\geq1\}.
\]
\end{definition}

As a consequence of Corollary \ref{cor:PoincSeries} one has
\begin{corollary} \label{cor:cx}
If $R=Q/I$ is a skew complete intersection, with $Q$ a skew polynomial ring in $n$ variables $x_1,\ldots,x_n$ and $I=(f_1,\ldots,f_c)\subseteq(x_1,\ldots, x_n)^2$, then $\mathrm{cx}_R\;\kk=c$.
\end{corollary}

The next corollary answers \cite[Question 6.5]{NnCommCI} for skew complete intersections:
\begin{corollary}
If $R$ is a skew complete intersection then $\Ext_R(\kk,\kk)$ is a noetherian algebra.
\end{corollary}
\begin{proof}
A presentation of $\Ext_R(\kk,\kk)$ is given by Theorem \ref{thm:ExtPres}. It is clear from that presentation that this algebra is finitely generated over the subring generated by $\theta_{n+1},\ldots,\theta_{n+c}$, which is a skew polynomial ring and hence noetherian. The result now follows.
\end{proof}
\begin{definition}\cite{CasShel}
Let $A$ be a $\kk$-algebra which is finitely generated in degree 1. One says that $A$ is a $\mathcal{K}_2$-algebra if $\Ext_A(\kk,\kk)$ is generated as an algebra by $\Ext_A^1(\kk,\kk)$ and $\Ext_A^2(\kk,\kk)$.
\end{definition}
The notion of a $\mathcal{K}_2$-algebra is a generalization of the notion of 
Koszul algebra which has been studied in the literature. As a consequence of 
Theorem \ref{thm:ExtPres} we deduce the following result generalizing
\cite[Corollary 9.2]{CasShel}.
\begin{corollary}
If $R$ is a skew complete intersection generated in degree one then $R$ is a $\mathcal{K}_2$-algebra.
\end{corollary}

\bibliographystyle{amsplain}

\bibliography{biblio}

\providecommand{\bysame}{\leavevmode\hbox to3em{\hrulefill}\thinspace}
\providecommand{\MR}{\relax\ifhmode\unskip\space\fi MR }
\providecommand{\MRhref}[2]{%
  \href{http://www.ams.org/mathscinet-getitem?mr=#1}{#2}
}
\providecommand{\href}[2]{#2}
\begin{thebibliography}{10}

\bibitem{IFR}
Luchezar~L. Avramov, \emph{Infinite free resolutions}, Six lectures on
  commutative algebra ({B}ellaterra, 1996), Progr. Math., vol. 166,
  Birkh\"{a}user, Basel, 1998, pp.~1--118. \MR{1648664}

\bibitem{AGP}
Luchezar~L. Avramov, Vesselin~N. Gasharov, and Irena~V. Peeva, \emph{Complete
  intersection dimension}, Inst. Hautes \'Etudes Sci. Publ. Math. (1997),
  no.~86, 67--114 (1998). \MR{1608565}

\bibitem{BerOpp}
Petter~Andreas Bergh and Steffen Oppermann, \emph{Cohomology of twisted tensor
  products}, J. Algebra \textbf{320} (2008), no.~8, 3327--3338. \MR{2450729}

\bibitem{CSV}
Andreas Cap, Hermann Schichl, and Ji\v{r}\'{i} Van\v{z}ura, \emph{On twisted
  tensor products of algebras}, Comm. Algebra \textbf{23} (1995), no.~12,
  4701--4735. \MR{1352565}

\bibitem{CasShel}
Thomas Cassidy and Brad Shelton, \emph{Generalizing the notion of {K}oszul
  algebra}, Math. Z. \textbf{260} (2008), no.~1, 93--114. \MR{2413345}

\bibitem{Feld}
J\"org Feldvoss, \emph{Representations of {L}ie colour algebras}, Adv. Math.
  \textbf{157} (2001), no.~2, 95--137. \MR{1813428}

\bibitem{Gol}
E.~S. Golod, \emph{Standard bases and homology}, Algebra---some current trends
  ({V}arna, 1986), Lecture Notes in Math., vol. 1352, Springer, Berlin, 1988,
  pp.~88--95. \MR{981820}

\bibitem{GS}
E.~S. Golod and I.~R. \v{S}afarevi\v{c}, \emph{On the class field tower}, Izv.
  Akad. Nauk SSSR Ser. Mat. \textbf{28} (1964), 261--272. \MR{0161852}

\bibitem{GulLev}
Tor~H. Gulliksen and Gerson Levin, \emph{Homology of local rings}, Queen's
  Paper in Pure and Applied Mathematics, No. 20, Queen's University, Kingston,
  Ont., 1969. \MR{0262227}

\bibitem{Gul}
Tor~Holtedahl Gulliksen, \emph{A proof of the existence of minimal
  {$R$}-algebra resolutions}, Acta Math. \textbf{120} (1968), 53--58.
  \MR{0224607}

\bibitem{Khar}
Vladislav Kharchenko, \emph{Quantum {L}ie theory}, Lecture Notes in
  Mathematics, vol. 2150, Springer, Cham, 2015, A multilinear approach.
  \MR{3445175}

\bibitem{KKZ}
E.~Kirkman, J.~Kuzmanovich, and J.~J. Zhang, \emph{Shephard-{T}odd-{C}hevalley
  theorem for skew polynomial rings}, Algebr. Represent. Theory \textbf{13}
  (2010), no.~2, 127--158. \MR{2601538}

\bibitem{NnCommCI}
\bysame, \emph{Noncommutative complete intersections}, J. Algebra \textbf{429}
  (2015), 253--286. \MR{3320624}

\bibitem{PnP}
Alexander Polishchuk and Leonid Positselski, \emph{Quadratic algebras},
  University Lecture Series, vol.~37, American Mathematical Society,
  Providence, RI, 2005. \MR{2177131}

\bibitem{Scheu}
M.~Scheunert, \emph{Generalized {L}ie algebras}, J. Math. Phys. \textbf{20}
  (1979), no.~4, 712--720. \MR{529734}

\bibitem{Scho}
Colette Schoeller, \emph{Homologie des anneaux locaux noeth\'{e}riens}, C. R.
  Acad. Sci. Paris S\'{e}r. A-B \textbf{265} (1967), A768--A771. \MR{0224682}

\bibitem{Ser}
Jean-Pierre Serre, \emph{Sur la dimension homologique des anneaux et des
  modules noeth\'{e}riens}, Proceedings of the international symposium on
  algebraic number theory, {T}okyo \& {N}ikko, 1955, Science Council of Japan,
  Tokyo, 1956, pp.~175--189. \MR{0086071}

\bibitem{Sjo}
Gunnar Sj\"{o}din, \emph{A set of generators for {${\rm Ext}_{R}(k,k)$}}, Math.
  Scand. \textbf{38} (1976), no.~2, 199--210. \MR{0422248}

\bibitem{Tate}
John Tate, \emph{Homology of {N}oetherian rings and local rings}, Illinois J.
  Math. \textbf{1} (1957), 14--27. \MR{0086072}

\bibitem{Wei}
Charles~A. Weibel, \emph{An introduction to homological algebra}, Cambridge
  Studies in Advanced Mathematics, vol.~38, Cambridge University Press,
  Cambridge, 1994. \MR{1269324}

\end{thebibliography}

\end{document}